\documentclass[a4paper]{amsart}

\usepackage[utf8]{inputenc}
\usepackage{amsthm, amssymb, amsmath, amsfonts}
\usepackage{url}
\usepackage{graphicx}

\newtheorem{thm}{Theorem}[section]
\newtheorem{prop}[thm]{Proposition}
\newtheorem{lem}[thm]{Lemma}
\newtheorem{cor}[thm]{Corollary}
\theoremstyle{remark}
\newtheorem{rem}[thm]{Remark}

\renewcommand{\le}{\leqslant}
\renewcommand{\ge}{\geqslant}
\renewcommand{\subset}{\subseteq}
\newcommand{\mcl}{\mathcal}

\newcommand{\E}{\mathbb{E}}
\newcommand{\EE}{\mathbf{E}}

\newcommand{\N}{\mathbb{N}}

\newcommand{\1}{\mathbf{1}}
\newcommand{\R}{\mathbb{R}}

\newcommand{\Z}{\mathbb{Z}}
\renewcommand{\P}{\mathbb{P}}
\newcommand{\PP}{\mathbf{P}}

\newcommand{\ov}{\overline}
\newcommand{\td}{\tilde}
\newcommand{\eps}{\varepsilon}
\def\d{{\mathrm{d}}}

\newcommand{\mfk}{\mathfrak}
\newcommand{\Ll}{\left}
\newcommand{\Rr}{\right}
\renewcommand{\phi}{\varphi}

\renewcommand{\b}{_\beta}
\renewcommand{\l}{_{l,\beta}}

\newcommand{\tdj}{\td{\jmath}}
\renewcommand{\emptyset}{\varnothing}

\title[Lyapunov exponents of random walks in small random potential]{Lyapunov exponents of random walks in small random potential: the lower bound}
\author{Thomas Mountford, Jean-Christophe Mourrat}

\address{Ecole polytechnique fédérale de Lausanne, institut de mathématiques, station 8, 1015 Lausanne, Switzerland}

\begin{document}
\begin{abstract}
We consider the simple random walk on $\Z^d$, $d \ge 3$, evolving in a potential of the form $\beta V$, where $(V(x))_{x \in \Z^d}$ are i.i.d.\ random variables taking values in~$[0,+\infty)$, and $\beta > 0$. When the potential is integrable, the asymptotic behaviours as $\beta$ tends to $0$ of the associated quenched and annealed Lyapunov exponents are known (and coincide). Here, we do not assume such integrability, and prove a sharp lower bound on the annealed Lyapunov exponent for small~$\beta$. The result can be rephrased in terms of the decay of the averaged Green function of the Anderson Hamiltonian $-\triangle + \beta V$.

\bigskip


\medskip


\end{abstract}
\maketitle
%\tableofcontents
%
%
%
%
%
%%%%%%%%%%%%%%%%%%%%%%%%%%%%%%%%%%%%%%%%%%%%%%%%%%%%%%%%%%%%%%
%%%%%%%%%%%%%%%%%%%%%%%%%%%%%%%%%%%%%%%%%%%%%%%%%%%%%%%%%%%%%%
%
%
%
\section{Introduction}
\label{s:intro}
\setcounter{equation}{0}

Let $(S_n)_{n \in \N}$ be the simple random walk on $\Z^d$, $d \ge 3$. We write $\PP_x$ for the law of the random walk starting from position $x$, and $\EE_x$ for the associated expectation. Independently of $S$, we give ourselves a family $(V(x))_{x \in \Z^d}$ of independent random variables, which we may call the \emph{potential}, or also the \emph{environment}. These random variables are distributed according to a common probability measure $\mu$ on $[0,+\infty)$. We write $\P = \mu^{\otimes \Z^d}$ for their joint distribution, and $\E$ for the associated expectation. Let $\ell \in \R^d$ be a vector of unit Euclidian norm, and
$$
T_n(\ell) = \inf \Ll\{ k : S_k \cdot \ell \ge n \Rr\}
$$ 
be the first time at which the random walk crosses the hyperplane orthogonal to~$\ell$ lying at distance $n$ from the origin. Our main goal is to study the \emph{quenched} and \emph{annealed} point-to-hyperplane \emph{Lyapunov norms} (also called \emph{Lyapunov exponents}), defined respectively by
\begin{equation}
\label{defalpha}
\alpha_\beta(\ell) = \lim_{n \to + \infty} - \frac{1}{n} \ \log \EE_0\Ll[ \exp\Ll( -\sum_{k = 0}^{T_n(\ell)-1} \beta V(S_k) \Rr) %\ \1_{\{T_n(\ell) < + \infty\}} 
\Rr],
\end{equation}
\begin{equation}
\label{defalphabar}
\ov{\alpha}_\beta(\ell) = \lim_{n \to + \infty} - \frac{1}{n} \ \log \E\EE_0\Ll[ \exp\Ll( -\sum_{k = 0}^{T_n(\ell)-1} \beta V(S_k) \Rr) %\ \1_{\{T_n(\ell) < + \infty\}}
 \Rr],
\end{equation}
for $\beta > 0$ tending to $0$ (the first limit holds almost surely; see \cite{shape, flu1} for proofs that these exponents are well defined). 

Intuitively, these two quantities measure the cost, at the exponential scale, of travelling from the origin to a distant hyperplane, for the random walk penalized by the potential $\beta V$. The quenched Lyapunov norm is a measure of this cost in a typical random environment, while the annealed Lyapunov norm measures this cost after averaging over the randomness of the environment. These norms are related to the point-to-point Lyapunov norms by duality, and to the large deviation rate function for the position of the random walk at a large time under the weighted measure (see \cite{flu1, shape} for details).

Recently, \cite{wang1,wang2, kmz} studied this question under the additional assumption that $\E[V]$ is finite (where we write $\E[V]$ as shorthand for $\E[V(0)]$). They found that, as $\beta$ tends to $0$,
\begin{equation}
\label{resultatkmz}
\alpha_\beta(\ell) \sim \ov{\alpha}_\beta(\ell) \sim \sqrt{2d \ \beta \ \E[V]}
\end{equation}
(and they showed that this relation also holds for $d \in \{1,2\}$). This means that when $\E[V]$ is finite, the first-order asymptotics of the Lyapunov exponents are the same as if the potential were non-random and uniformly equal to $\E[V]$.

Our goal is to understand what happens when we drop the assumption on the integrability of the potential. From now on, 
\begin{equation}
\label{non-integr}
\text{we assume that } \E[V] = + \infty,
\end{equation}
and write
\begin{equation}
\label{defen}
e_{\beta,n} = \E\EE_0\Ll[ \exp\Ll( -\sum_{k = 0}^{T_n(\ell)-1} \beta V(S_k) \Rr)  \Rr].
\end{equation}
Here is our main result.
\begin{thm}
\label{t:estimen}
Let $\eps > 0$. There exists $C > 0$ such that for any $\beta$ small enough and any $n$,
$$
e_{\beta,n} \le C \exp\Ll( -(1-\eps) \sqrt{2d \ \mfk{I}\b} \ n\Rr),
$$	
where 
\begin{equation}
\label{defmfkI}
\mfk{I}\b = q_d \int \frac{1-e^{-\beta z}}{1-(1-q_d)e^{-\beta z}} \ \d \mu(z),
\end{equation}
and $q_d$ is the probability that the simple random walk never returns to its starting point, that is,
\begin{equation}
\label{defqd}
q_d = \PP_0\Ll[\forall n \ge 1, S_n \neq 0\Rr].	
\end{equation}
\end{thm}
This result is a first step towards a proof that, as $\beta$ tends to $0$,
\begin{equation}
\label{conject}
\alpha_\beta(\ell) \sim \ov{\alpha}_\beta(\ell) \sim \sqrt{2d \ \mfk{I}\b}.
\end{equation}
One can check that $\ov{\alpha}_\beta(\ell) \le \alpha_\beta(\ell)$, and Theorem~\ref{t:estimen} provides the adequate lower bound on $\ov{\alpha}_\beta(\ell)$ for \eqref{conject} to hold. In order to complete the proof of \eqref{conject}, there remains to provide a matching upper bound for $\alpha_\beta(\ell)$. This will be done in a companion paper. 

\begin{rem}
\label{r:integr}
Let us write 
\begin{equation}
\label{deff}
f(z) = 	q_d \frac{1-e^{-z}}{1-(1-q_d)e^{-z}}  = \frac{q_d}{1-q_d} \Ll( 1-\frac{q_d}{1-(1-q_d)e^{-z}} \Rr), %\\
%= \frac{q_d}{1-q_d}\Ll( 1-\sum_{k = 0}^{+\infty} (1-q_d)^k e^{-kz} \Rr)
\end{equation}
so that 
$$
\mfk{I}\b = \int f(\beta z) \ \d \mu(z).
$$
It is easy to see that $f$ is concave increasing and that $f(z) \sim z$ as $z$ tends to $0$. As a consequence, for any $M >0$,
\begin{equation}
\label{asymptintegr}
\int_{z \le M} f(\beta z) \ \d \mu(z) \sim \beta \ \E[V \1_{V \le M}],
\end{equation}
while, since $f(z) \le z$,
$$
\int_{z > M} f(\beta z) \ \d \mu(z) \le \beta \ \E[V \1_{V > M}].
$$
When $\E[V]$ is finite, we thus obtain that the right-hand sides of (\ref{resultatkmz}) and (\ref{conject}) are equivalent as $\beta$ tends to $0$, and thus \eqref{conject} holds indeed in this case.
\end{rem}

\begin{rem}
An interesting feature of \eqref{resultatkmz} and \eqref{conject} is that their right-hand sides do not depend on $\ell$. In other words, asymptotically, the balls associated to the quenched and annealed Lyapunov norms look like scaled Euclidian balls.
\end{rem}

The main motivation behind \cite{wang1} was related to questions concerning the spectrum of the discrete Anderson Hamiltonian $H\b = -\triangle + \beta V$, where $\triangle$ is the discrete Laplacian:
\begin{equation}
\label{deftriangle}
\triangle f (x) = \frac{1}{2d}\sum_{y \sim x} (f(y) - f(x)).
\end{equation}
Powerful techniques have been devised to transfer finite-volume estimates on the Green function of $H\b$ within some energy interval into information on the spectrum in this interval (see \cite{fs,fmss,dk} for the multiscale method, and \cite{am,afhs} for the fractional-moment approach). For instance, it is known that for any $\beta > 0$, the spectrum of $H\b$ is pure point in a neighbourhood of $0$ and corresponding eigenfunctions are exponentially localized. In \cite{wangloc}, extending the techniques developed in \cite{wang1}, the author gave quantitative estimates on the Green function within an explicit energy interval at the edge of the spectrum, as~$\beta$ tends to $0$. These were then refined in \cite{klopp}. These results imply in particular that if $\E[V]$ is finite and the distribution $\mu$ is absolutely continuous with respect to the Lebesgue measure, then for any $\eta > 0$ and any~$\beta$ small enough, the spectrum of $H\b$ is pure point in the interval $[0,\beta\E[V]-\beta\eta]$, with exponentially decaying eigenfunctions.
% (not tuned to the normalization in 1/(2d), and moreover, any eigenfunction corresponding to an eigenvalue $E$ in this interval decays exponentially fast with an exponent of decay at least $c \sqrt{\beta\E[V] - E}$ for some fixed $c > 0$ (independent of~$\beta$).
Theorem~\ref{t:estimen} can be seen as a first step towards a study of these questions in the case when the potential is not assumed to be integrable. We conjecture that when this integrability condition is dropped, the upper energy $\beta \E[V]$ appearing in the above result should be relaced by
\begin{equation}
\label{conj:int}
\int \Ll( \frac{1}{q_d} + \frac{1}{\beta z} \Rr)^{-1} \ \d \mu(z).
\end{equation}
The reason why this is the natural integral to consider will be explained in Section~\ref{s:extension}.

\medskip

We now give a heuristic description of the typical scenario responsible for the behaviour of $e_{\beta,n}$ described in Theorem~\ref{t:estimen}. Different strategies can be used to reduce the cost of travel to the distant hyperplane. (1) One approach is to reach the hyperplane in a small number of steps. (2) A second approach is to avoid sites where $V(x)$ is too large, or else, to try not to return to such sites too many times. Naturally, one should look for the optimal strategy as a combination of these two methods. 

In order to quantify method (1), one can observe that, for small $v$, 
\begin{equation}
\label{met1}
- \log \PP_0[T_n(\ell) \approx n/v] \approx  \frac{d v}{2} \ n.	
\end{equation}
The quantity $v$ represents the velocity of the particle.
%(* Statement (\ref{met1}) can be obtained (and made precise) through a large deviation estimate concerning the sum of the i.i.d.\ random variables $(S_{k+1}-S_k) \cdot \ell$. *)
On the other hand, roughly speaking, we will show that, for small $v$,
\begin{equation}
\label{met2}
- \log \E\EE_0\Ll[ \exp\Ll( -\sum_{k = 0}^{T_n(\ell)-1} \beta V(S_k) \Rr) \ \Big| \ T_n(\ell) \approx n/v  \Rr] \\
\gtrsim   \mfk{I}\b \ \frac{n}{v},	
\end{equation}
which quantifies the gains obtained by method (2). 

Assuming that these observations hold, Theorem~\ref{t:estimen} can be derived by optimizing~$v$ so that the sum of the costs in (\ref{met1}) and (\ref{met2}) is minimized. This is achieved choosing 
\begin{equation}
\label{conj:v}
v = \sqrt{\frac{2}{d} \ \mfk{I}\b}.
\end{equation}
%In order to get a rough understanding of this formula, it is interesting to imagine that the $V(S_k)$ appearing in the formulas (\ref{defalpha}) and (\ref{defalphabar}) are replaced by i.i.d.\ random variables distributed according to $\mu$. Another way to put it would be to say that the environment is refreshed at every single step of the random walk (this corresponds to the physicist's use of the term ``annealed'', as for instance in \cite[Paragraph~1.2.3.1]{bougeo}). In this fictitious case, we would have
%$$
%\log \E\EE_0\Ll[ \exp\Ll( -\sum_{k = 0}^{n} \beta V(S_k) \Rr) \Rr] = 
%$$

Let us explain the meaning of \eqref{met2}. Recall that for any $M$, \eqref{asymptintegr} holds. Relation \eqref{met2} shows that sites whose potential is bounded by $M$ contribute as if they were replaced by their expectation. In other words, for these sites, method~(2) is simply too costly to be effective, and we may say that these sites are in a ``law of large numbers'' regime. In fact, in this reasoning, we could allow $M$ to grow with~$\beta$, as long as it remains small compared to $\beta^{-1}$.

The picture changes when we consider sites whose potential is very large compared to $\beta^{-1}$. Observe that the number of distinct sites visited by the random walk at time $n/v$ grows like $q_d  n/v$, where $q_d^{-1}$ is the mean number of visits to any point visited (conditionally on the event $T_n(\ell) \approx n/v$, this is true in the limit of small~$v$). Under the annealed measure, the sequence of potentials attached to the distinct sites visited forms a sequence of i.i.d.\ random variables with common distribution~$\mu$. The cost of not meeting any site whose potential lies in the interval $[\beta^{-1} M,+\infty)$ up to time $T_n(\ell) \approx n/v$ is thus approximately
$$
-\log \Ll(\Ll(1-\mu\Ll( [\beta^{-1} M,+\infty) \Rr)\Rr)^{q_d n/v}\Rr) \approx q_d \ \mu\Ll( [\beta^{-1} M,+\infty)\Rr) \ \frac{n}{v}.
$$
When $M$ is large, since $f(z) \to q_d$ as $z$ tends to infinity, this is roughly 
$$
\int_{z \ge \beta^{-1} M} f(\beta z) \ \d \mu(z) \ \frac{n}{v},
$$
and thus the strategy concerning sites whose potential is much larger than $\beta^{-1}$ is simple to describe: simply avoid them. 

To sum up, formula (\ref{met2}) reveals the following picture. Sites whose potential is much smaller than $\beta^{-1}$ stay in a law of large numbers regime. Sites whose potential is much greater than $\beta^{-1}$ are simply avoided. Now, for sites whose potential is of order $\beta^{-1}$, an adequate intermediate strategy is found. Heuristically, for sites whose potential is roughly $\beta^{-1} z$, the strategy consists in (i) lowering the proportion of such sites that are visited by a factor $(1-e^{-z})/z$ ; (ii) once on such a site, to go back to it with a probability $(1-q_d) e^{-z}$ instead of $(1-q_d)$.

\medskip

The picture described above, and in particular \eqref{met2}, must however be taken with a grain of salt. Our basic approach relies on coarse-graining arguments. We identify \emph{good} boxes, which are such that we understand well the cost and the exit distribution of a coarse-grained displacement of the walk started from a point in a good box. In our arguments, we do not try to control what happens when a coarse-grained piece of trajectory starts within a \emph{bad} box. As was noted in \cite{sz}, this is indeed a delicate matter, since the time spent in bad boxes does not have finite exponential moments in general. Instead, we introduce a \emph{surgery} on the trajectories. The surgery consists in removing certain \emph{loops}, which are pieces of coarse-grained trajectory that start and end in the same bad box. We show rigorous versions of \eqref{met1} and \eqref{met2}, where $T_n(\ell)$ is replaced by the total time spent outside of these loops; from these estimates, we then derive Theorem~\ref{t:estimen}.

\medskip

\noindent \textbf{Related works.} We already mentioned \cite{wang1, wangloc,wang2,kmz} and the connection with Anderson localization. 

In \cite{iv12}, it is proved that under the annealed weighted measure, the random walk conditioned to hit a distant hyperplane satisfies a law of large numbers (see also \cite{sz,km}). It would be interesting to see whether their techniques can be combined with our present estimate to show that indeed, the right-hand side of~\eqref{conj:v} gives the asymptotic behaviour of the speed as $\beta$ tends to $0$ (our results do not show this directly, due to the surgery on paths discussed above).

Another motivation relates to recent investigation on whether the disorder is \emph{weak} of \emph{strong}. The disorder is said to be weak if the quenched and annealed Lyapunov exponents coincide. To our knowledge, this question has only been investigated for potentials of the form $\lambda + \beta V$ with $\lambda > 0$, see \cite{flu2,zyg1,iv} for weak disorder results when $d \ge 4$ and $\beta$ is small, and \cite{zyg2} for strong disorder results when $d \le 3$. This additional $\lambda > 0$ is very convenient since it introduces an effective drift towards the target hyperplane (indeed, the problem can be rewritten in terms of a drifted random walk in the potential $\beta V$ using a Girsanov transform). In particular, the asymptotic speed of travel to the hyperplane remains bounded away from $0$ in this case. One of our motivations was to get a better understanding of the behaviour of the walk when we set $\lambda = 0$. Of course, showing that the Lyapunov exponents are equivalent as $\beta$ tends to $0$ does not touch upon the question whether they become equal for small $\beta$ or not. 
%For a more thorough discussion of this and related questions, we refer the reader to the review \cite{ivreview}.

Recently, a continuous-space version of \cite{kmz} was obtained in \cite{ru}. There, the author investigates Brownian motion up to reaching a unit ball at distance $n$ from the starting point, and evolving in a potential formed by locating a given compactly supported bounded function $W_n$ at each point of a homogeneous Poisson point process of intensity $\nu_n$. It is shown that if $\nu_n \|W_n\|_1 \sim D/n$ for some constant~$D$, then the quenched and annealed Lyapunov exponents are both asymptotically equivalent to $\sqrt{2D/n}$.

\medskip

\noindent \textbf{Organization of the paper.} As was apparent in the informal description above, the most interesting phenomena occur for sites whose associated potential is of the order of $\beta^{-1}$. Section~\ref{s:lowerbd} adresses this case, and proves Theorem~\ref{t:estimen} with $\mfk{I}\b$ replaced by
\begin{equation*}
%\label{intdroite}
\mcl{I}\b = \int_{\beta z \ge a} f(\beta z) \ \d \mu(z),
\end{equation*}
where $a > 0$ is arbitrary. This is however not sufficient to prove Theorem~\ref{t:estimen} in full generality, since for some distributions [and although we always assume \eqref{non-integr}], it may be that whatever $a > 0$, the integral $\mcl{I}\b$ is too small compared to $\mfk{I}\b$ as $\beta$ tends to $0$. In other words, there are cases for which the integral
$$
\td{\mcl{I}}\b = \int_{\beta z < a} f(\beta z) \ \d \mu(z)
$$
cannot be neglected, even if we are free to choose $a > 0$ beforehand. In the very short Section~\ref{s:intermediate}, we take a specific choice for $a$ and distinguish between three cases, depending on whether $\mcl{I}\b$, $\td{\mcl{I}}\b$, or both integrals have to be considered. Section~\ref{s:intermediate2} tackles the most delicate case when both integrals must be accounted for. Section~\ref{s:intermediate_only} concludes the proof of Theorem~\ref{t:estimen}, covering the case when $\mcl{I}\b$ is negligible compared to $\td{\mcl{I}}\b$. Finally, Section~\ref{s:extension} presents natural extensions of Theorem~\ref{t:estimen}, and spells out the link with the Green function of the operator $-\triangle + \beta V$.

\medskip

\noindent \textbf{Notations.} We write $|\cdot|$ for the Euclidian norm on $\R^d$. For $x \in \R^d$ and $r > 0$, let
\begin{equation}
\label{defDxr}
D(x,r) = \{y \in \R^d : |y-x| \le r\}
\end{equation}
be the ball of centre $x$ and radius $r$. For $x \in \Z^d$ and $r \in \N$, we call $\emph{box}$ of centre $x$ and size $r$ the set defined by
\begin{equation}
\label{defBxr}
B(x,r) = x + \{-r,\ldots,r\}^d.	
\end{equation}
For $A \subset \Z^d$, we write $|A|$ to denote the cardinality of $A$.

%
%
%
%
%%%%%%%%%%%%%%%%%%%%%%%%%%%%%%%%%%%%%%%%%%%%%%%%%%%%%%%%%%%%%%
%%%%%%%%%%%%%%%%%%%%%%%%%%%%%%%%%%%%%%%%%%%%%%%%%%%%%%%%%%%%%%
%
%
%
\section{The contribution of important sites}
\label{s:lowerbd}
\setcounter{equation}{0}

Let $\eps > 0$. This $\eps$ will play the role of the allowed margin of error in our subsequent reasoning. We will need to assume that is is sufficiently small (and we may do so without mentioning it explicitly), but it will be kept fixed throughout.

\subsection{Splitting the interval}

As a start, we fix $a > 0$ and focus on sites whose associated potential lies in the interval $[\beta^{-1}a,+\infty)$. 

We want to approximate the integral
\begin{equation}
\label{defmclI}
\mcl{I}\b = \int_{\beta z \ge a} f(\beta z) \ \d \mu(z)
\end{equation}
by a Riemann sum (recall that the function $f$ was defined in (\ref{deff})). Let $\kappa'$ be a positive integer and $a = a_0' < a_1' < \cdots < a_{\kappa'}' = +\infty$ be such that for every~$l$, one has $(1-\eps) f(a_{l+1}') \le f(a_l')$. This provides us with a subdivision of the interval $[\beta^{-1} a,+\infty)$, and 
\begin{equation}
\label{finesplit}
(1-\eps) \ \mcl{I}\b \le  \sum_{l=0}^{\kappa' - 1} f(a_l') \ \mu \Ll( [\beta^{-1} a_l',\beta^{-1} a_{l+1}') \Rr).
\end{equation}
From this subdivision, we extract those intervals which have a non-negligible weight:
$$
\mcl{L} = \Ll\{[a_l',a_{l+1}') : \mu \Ll( [\beta^{-1} a_l',\beta^{-1} a_{l+1}') \Rr) \ge \frac{\eps}{\kappa'} f(a) \mu \Ll( [\beta^{-1} a,+\infty) \Rr) \Rr\}.
$$
We let $\kappa > 0$ denote the cardinality of $\mcl{L}$, and let $a_1 < b_1 \le a_2 < b_2 \le \cdots \le a_\kappa < b_\kappa$ be such that 
$$
\mcl{L} = \Ll\{[a_l,b_l), 1 \le l \le \kappa \Rr\}.
$$
Although $\kappa$, $a_l$ and $b_l$ may depend on $\beta$, we keep this dependence implicit in the notation. (Since $\kappa$ remains bounded by $\kappa'$, this dependence will not be a problem.) 
Noting that
$$
\sum_{l=0}^{\kappa' - 1} f(a_l') \ \mu \Ll( [\beta^{-1} a_l',\beta^{-1} a_{l+1}') \Rr) \ge f(a) \ \mu \Ll( [\beta^{-1} a,+\infty) \Rr),
$$
and that $f$ is bounded by $q_d \le 1$, we obtain that
$$
\sum_{l: [a_l',a_{l+1}') \notin \mcl{L}} f(a_l') \ \mu \Ll( [\beta^{-1} a_l',\beta^{-1} a_{l+1}') \Rr) \le \eps \sum_{l=0}^{\kappa' - 1} f(a_l') \ \mu \Ll( [\beta^{-1} a_l',\beta^{-1} a_{l+1}') \Rr),
$$
and thus, letting
\begin{equation}
\label{defIb}
I\b = \sum_{l=1}^{\kappa} f(a_l) \ \mu \Ll( [\beta^{-1} a_l,\beta^{-1} b_{l}) \Rr),
\end{equation}
we are led to
\begin{equation}
\label{compIb}
I\b \ge (1-\eps) \sum_{l=0}^{\kappa' - 1} f(a_l') \ \mu \Ll( [\beta^{-1} a_l',\beta^{-1} a_{l+1}') \Rr) \stackrel{\text{\eqref{finesplit}}}{\ge} (1-\eps)^2 \  \mcl{I}\b.
\end{equation}
Up to a multiplicative error controlled by $\eps$, we can thus consider $I\b$ as a good approximation of $\mcl{I}\b$. We let
$$
T_l = \{x \in \Z^d : \beta V(x) \in [a_l,b_l) \}, \qquad T = \bigcup_{l = 1}^\kappa T_l.
$$
We call elements of $T$ \emph{important sites}. A relevant length scale of our problem is $\hat{L}\b$ defined by
\begin{equation}
\label{defhatLb}
\hat{L}\b^{-2} = \P\Ll[0 \in T \Rr].	
\end{equation}
This scale is interesting since it is such that, if the random walk runs up to a distance $\hat{L}\b$ from its starting point, then it meets roughly one important site. Clearly, $\hat{L}\b$ tends to infinity as $\beta$ tends to $0$. Let also
\begin{equation}
\label{defpl}
p_l = \mu \Ll( [\beta^{-1} a_l,\beta^{-1} b_{l}) \Rr) = \P\Ll[ \beta V \in [a_l,b_l) \Rr]  = \P\Ll[0 \in T_l \Rr].
\end{equation}
Although this is not explicit in the notation, $p_l$, $T_l$, and $T$ depend on~$\beta$. Note that
\begin{equation}
\label{comparLbp}
\hat{L}\b^{-2} = \sum_{l=1}^\kappa p_l \quad \text{ and } \quad p_l \ge \frac{\eps f(a)}{\kappa'} \ \hat{L}\b^{-2}.	
\end{equation}
We may write $p_l \simeq \hat{L}\b^{-2}$ to denote the fact that there are constants $C_1, C_2$ such that $C_1 \hat{L}\b^{-2} \le p_l \le C_2 \hat{L}\b^{-2}$. Similarly, we have
\begin{equation*}
%\label{compLb}
I\b\le \hat{L}\b^{-2} \le \frac{I\b}{f(a)},
\end{equation*}
so $I\b \simeq \hat{L}\b^{-2}$.

Although not really necessary, arguments developed in Section~\ref{s:intermediate2} will be clearer if instead of $\hat{L}\b$, we choose from now on $L\b$ such that
\begin{equation}
\label{defLb}
L\b^{-2} =  f(a) \ \hat{L}\b^{-2}
\end{equation}
as our length scale of reference, so that $L\b^{-2} \le I\b$. Note that $L\b \simeq \hat{L}\b$.

\subsection{A coarse-grained picture}
\label{ss:coarse-grained}
Let $R\b$ be a positive integer, which we refer to as the \emph{mesoscopic scale}. We define a coarse-graining of the trajectory at this scale. That is, we let $j_0 = 0$, and define recursively 
\begin{equation}
\label{defjn}
j_{n+1} = \inf \{k > j_n : S_k \notin D(S_{j_n}, R\b) \}.	
\end{equation}
%We write $X_n = S_{j_n}$. Note that $(X_n)$ is a random walk.

We will need to say that, most of the time, the values of the potential around the position of the random walk are ``typical''. For $i\in \Z^d$, let $B_i = B((2 R\b+1)i,R\b)$. The boxes $(B_i)_{i \in \Z^d}$ form a partition of $\Z^d$ at the mesoscopic scale. Roughly speaking, we will ask a ``nice'' box to contain sufficiently important sites that are not too close from one another, and are evenly spread across the box. 

In order to make this informal description precise, we introduce two additional scales $r\b, r\b'$ such that $r\b' < r\b < R\b$. We ask that we can partition a box $B$ of size $R\b$ by subboxes of size $r\b$ (that is, we ask $(2R\b +1)$ to be a multiple of $(2r\b + 1)$), and write $\mathcal{P}_i$ for the partition of $B_i$ into subboxes of size $r\b$. Similarly, we ask that any box $b$ of size $r\b$ can be partitioned into subboxes of size $r\b'$, and write $\mathcal{P}'_b$ for this partition. 

Let $\eps_1 < \eps/2d$, $b' = B(x,r\b')$, and $l \le \kappa$. If there exists $y \in B(x,(1-\eps_1)r\b')$ such that $y \in T_l$, and if moreover $y$ is the only important site inside $B(x,(1-\eps_1)r\b')$, then we define $\1(b',l) = 1$. Otherwise, we set $\1(b',l) = 0$. In other words, we have
\begin{equation}
\label{redef1b'l}
\1(b',l) = 1 \text{ iff } \Ll|B(x,(1-\eps_1)r\b') \cap T_l\Rr| = \Ll|B(x,(1-\eps_1)r\b') \cap T \Rr| = 1.
\end{equation}
The value of $\eps_1$ is chosen so that
\begin{equation}
\label{compvolumes}
\Ll|B(x,(1-\eps_1)r\b')\Rr| \ge \Ll(1-\frac{\eps}{2}\Rr)  |b'|.	
\end{equation}
\begin{equation}
\label{defbalanced}
\begin{array}{l}
\text{We say that a box } B_i \text{ is \emph{balanced} if for any box } b \in \mathcal{P}_i \text{ (of size } r\b\text{)} \\
\text{and any } l \le \kappa\text{, one has }
\sum_{b' \in \mcl{P}'_b} \1(b',l) \ge (1-\eps) p_l \ |b|.
\end{array}
\end{equation}
Observe that the event that $B_i$ is balanced depends only on the values of the potential inside the box $B_i$. We say that the box $B_i$ is \emph{good} if for any $j$ such that $\|j-i\|_\infty \le 1$, the box $B_j$ is balanced ; we say that it is a \emph{bad} box otherwise. The construction is summarized in Figure~\ref{f:boxes}.

\begin{figure}
\centering
\includegraphics[scale=0.7]{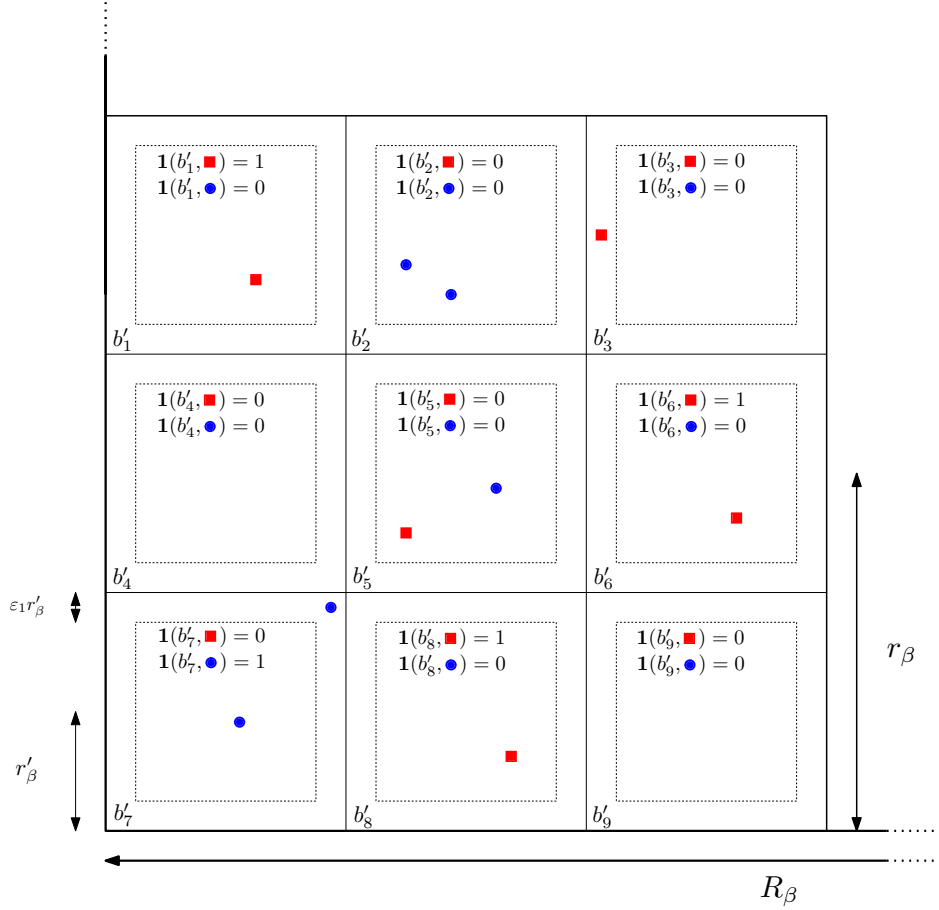}
\caption{
\small{
On this two-dimensional drawing, a box of size $r_\beta$ contains $9$ boxes of size $r_\beta'$. Two types of important sites are considered, marked by blue circles and red squares. Letting $b$ be the box of size $r_\beta$ depicted, we have $\mcl{P}_b' = \{b_1', \ldots, b_9'\}$.
}
}
\label{f:boxes}
\end{figure}

\subsection{Choosing the right scales}

First, we want to ensure that the walk $S$ does not visit any important site during a typical mesoscopic displacement. That is to say, we want $R\b \ll L\b$ (by this, we mean $L\b/R\b \to +\infty$ as $\beta \to 0$). 

While of course we ask $r\b \ll R\b$, we will need to have a growing number of important sites inside the intermediate boxes of size $r\b$. In order for this to be true, we need $r\b \gg L\b^{2/d}$. 

On the contrary, we want a typical box of size $r\b'$ to contain no important site, that is, $r\b' \ll L\b^{2/d}$. To sum up, we require
\begin{equation}
\label{compscales}
1 \ll r\b' \ll L\b^{2/d} \ll r\b \ll R\b \ll L\b.	
\end{equation}
It is convenient to make a specific choice regarding these scales. For $\delta > 0$, we define
\begin{equation}
\label{defscales}
r\b' = L\b^{2/d-\delta}, \quad r\b = L\b^{2/d + \delta}, \quad R\b = L\b^{1-\delta},	
\end{equation}
and fix $\delta$ small enough so that (\ref{compscales}) holds (this is possible since $d \ge 3$). An additional requirement on the smallness of $\delta$ will be met during the proof, see \eqref{conddelta}.

Thus defined, the scales may not satisfy our requirement that $(2R\b +1)$ be a multiple of $(2r\b + 1)$, and that $(2r\b +1)$ be a multiple of $(2r\b'+1)$. It is however easy to change the definitions slightly while preserving the asymptotics of the scales, and we will thus no longer pay attention to this problem.

\subsection{Most boxes are good}
We start by recalling some classical large deviations results about Bernoulli random variables. For $p \in [0,1]$ and $\eta \in [-p,1-p]$, we define
\begin{eqnarray}
\label{varform}
\psi_p(\eta) &  = & \sup_{\lambda} \Ll[ \lambda (p+\eta) - \log(1-p+pe^\lambda )  \Rr] \\
& = & (p+\eta) \log\Ll( \frac{p+\eta}{p} \Rr) + (1-p-\eta) \log\Ll( \frac{1-p-\eta}{1-p} \Rr) \notag.
\end{eqnarray}
We let $\psi_p(\eta) = +\infty$ when $\eta \notin [-p,1-p]$.
If $Y_1,Y_2,\ldots$ are independent Bernoulli random variables of parameter $p$ under the measure $\P$ (with the convention $ p = \P[Y_1 = 1]$), then for any $\eta \ge 0$, one has
\begin{equation}
\label{bernoullige}
\P\Ll[ \sum_{i = 1}^n Y_i \ge (p+\eta) n \Rr] \le e^{-n \psi_p(\eta)},
\end{equation}
and 
\begin{equation}
\label{bernoullile}
\P\Ll[ \sum_{i = 1}^n Y_i \le (p-\eta) n \Rr] \le e^{-n \psi_p(-\eta)}.
\end{equation}

A simple calculation shows that, for $\eta \in [-p,1-p]$,
$$
\psi_p''(\eta) = \frac{1}{(p+\eta)(1-p-\eta)}.
$$
\begin{equation}
\label{obs1}
\begin{array}{l}
\text{In particular, if } \eta \ge 0 \text{ is such that } p-\eta \ge 0,\text{ we note that } \\ 
\psi_p''(-\eta) \ge 1/p, \text{ and thus } \psi_p(-\eta) \ge \eta^2/(2p).
\end{array}	
\end{equation}

We now proceed to show that the probability for a box to be good is very close to $1$. 

\begin{lem}
\label{l:subsubbox}	
For $b'$ a box of size $r\b'$, $l \le \kappa$ and $\beta$ small enough, one has
$$
\P[\1(b',l) = 1] \ge \Ll(1-\frac{3\eps}{4}\Rr) p_l \ |b'|,
$$
where $p_l$ was defined in \eqref{defpl}.
\end{lem}
\begin{proof}
Without loss of generality, we may assume that the box $b'$ is centred at the origin. By the inclusion-exclusion principle, we have
\begin{multline}
\label{e:subsubbox1}
\P[B(0,(1-\eps_1)r\b') \cap T_l \neq \emptyset] \\
\ge \sum_{x \in B(0,(1-\eps_1) r\b')} \P[x \in T_l] - \sum_{\substack{x, y \in B(0,(1-\eps_1) r\b') \\ x < y}} \P[x,y \in T_l],	
\end{multline}
where, say, $x < y$ refers to lexicographical order. The first sum is equal to 
$$
p_l \ |B(0,(1-\eps_1) r\b')| \stackrel{\text{\eqref{compvolumes}}}{\ge} \Ll(1 - \frac{\eps}{2} \Rr) p_l \ |b'|,
$$
while the second sum in \eqref{e:subsubbox1} is bounded by $(p_l \ |b'|)^2$. We have seen in (\ref{comparLbp}) that $p_l \simeq L\b^{-2}$, and since $r\b' \ll L\b^{2/d}$, we have that $p_l \ |b'|$ tends to $0$ as $\beta$ tends to $0$, so the right-hand side of (\ref{e:subsubbox1}) is equal to 
$$
\Ll(1 - \frac{\eps}{2} \Rr) p_l \ |b'| (1 + o(1)).
$$
To conclude, it suffices to show that the probability of having two or more important sites within $B(x,(1-\eps_1)r\b')$ is negligible (see \eqref{redef1b'l}). But this is true since the probability is bounded by $(L\b^{-2} |b'|)^2 \ll p_l |b'|$.
\end{proof}
\begin{lem}
\label{p:subbox}
Let $b$ be a box of size $r\b$, and $l \le \kappa$. There exists $c_0 > 0$ (depending only on $\eps$) such that for any $\beta$ small enough,
$$
\P\Ll[ \sum_{b' \in \mcl{P}'_b} \1(b',l) < (1-\eps) p_l \ |b| \Rr] \le e^{-c_0 p_l  |b|}.
$$
\end{lem}
\begin{proof}
From Lemma~\ref{l:subsubbox} and the fact that $(\1(b',l))_{b' \in \mcl{P}'_b}$ are independent random variables, we know that $(\1(b',l))_{b' \in \mcl{P}'_b}$ stochastically dominate i.i.d.\ Bernoulli random variables with parameter 
$$
p = \Ll(1-\frac{3\eps}{4}\Rr) p_l \ |b'|.
$$ 
We are thus in the situation of inequality \eqref{bernoullile}, with 
$$
\eta = \frac{\eps}{4} p_l \ |b'| \quad \text{ and } \quad n = \frac{|b|}{|b'|}.
$$
The proposition follows using the observation made in \eqref{obs1}.
\end{proof}

\begin{prop}
\label{p:probbadbox}
There exists $c_1 > 0$ such that for $\beta$ small enough and any $i \in \Z^d$,
\begin{equation}
\label{probbadbox}
\P[B_i \text{ is not balanced}] \le \frac{|B_i|}{|b|} \ e^{-c_1 L\b^{-2} \ |b|},
\end{equation}
where $b$ is a box of size $r\b$.
\end{prop}
\begin{proof}
Considering the definition of balanced boxes given in \eqref{defbalanced}, one obtains the result with Lemma~\ref{p:subbox}, a union bound and the fact  that $p_l \simeq L\b^{-2}$.
\end{proof}

We say that $i,j \in \Z^d$ are $*$\emph{-neighbours}, and write $i \stackrel{*}{\sim} j$, if $\| i-j \|_\infty = 1$. We say that a subset of $\Z^d$ is $*$\emph{-connected} if it is connected for this adjacency relation. We call \emph{lattice animal} a subset of $\Z^d$ which is $*$-connected and contains the origin. 

We will be interested in the set of $i$'s such that $B_i$ is visited by the coarse-grained trajectory, which indeed forms a lattice animal if the walk is started at the origin. 
%The size of a lattice animal corresponding to a trajectory reaching a hyperplane at distance $n$ from the origin must be larger than some constant times $n/R\b$. With this a priori bound in mind, we want to show that the cost of having a lattice animal with a high proportion of bad boxes is much greater than $L\b^{-1} \simeq \sqrt{I\b}$. 
The next proposition 
%does precisely this, and 
is similar to an argument found in \cite[p.~1009]{sz}.
\begin{prop}
\label{animals}
Recall the definition of the scales in terms of $\delta$ given in \eqref{defscales}, and let $\eta\b = L\b^{-5 \delta/2}$. For any $\beta$ small enough, 
\begin{equation}
\label{e:animals}
\P\left[
\begin{array}{c}
\exists A \subset \Z^d : A \text{ is a lattice animal, } |A| \ge N \text{ and } \\
\Ll| \{ i \in A : B_i \text{ is a bad box}\} \Rr| \ge \eta\b \ |A|
\end{array}
\right]  \le e^{-N}.
%\\ \le C \exp\Ll(-  L\b^{ -(1-\gamma)} R\b N\Rr).
\end{equation}
\end{prop}
\begin{proof}
It suffices to show that, for some $c > 0$ and for $\beta$ small enough,
\begin{equation}
\label{e:animals0}
\P\left[ 
\begin{array}{c}
\exists A \subset \Z^d : A \text{ is a lattice animal, } |A| = N \text{ and } \\
\Ll| \{ i \in A : B_i \text{ is a bad box}\} \Rr| \ge \eta\b N
\end{array}
\right] \le \exp\Ll(-c L\b^{\delta/4} N\Rr).
\end{equation}

First, as observed in \cite[p.~1009]{sz}, there are at most $3^{2dN}$ lattice animals of cardinality $N$ (to see this, one can encode the lattice animal by a $*$-nearest-neighbour trajectory starting from the origin, of length at most $2N$, that is the ``depth-first search'' of a spanning tree of the lattice animal). Now, given $c_2 > 0$ and a lattice animal $A$ of cardinality $N$, the probability
\begin{equation}
\label{e:animals2}
\P\big[ \Ll| \{ i \in A : B_i \text{ is not balanced}\} \Rr| \ge c_2 \eta\b  N \big]	
\end{equation}
is of the form of the left-hand side of (\ref{bernoullige}), with, according to Proposition~\ref{p:probbadbox}, 
$$
p = \P[B_i \text{ is not balanced}] \le \frac{|B_i|}{|b|} \ e^{-c_1 L\b^{-2} \  |b|},
$$
and 
$$
\eta = c_2 \eta\b - p \sim c_2 \eta\b.
$$
With these parameters, we obtain that $\psi_p(\eta) \sim \eta \log(1/p)\sim c_1 c_2 \eta\b L\b^{-2} |b|$. Recalling that $|b| \simeq L\b^{2+d\delta}$, we infer that for some $c > 0$ and $\beta$ small enough, the probability in (\ref{e:animals2}) is smaller than $\exp(-c L\b^{\delta/2} N)$. To conclude, note that an unbalanced box can be responsible for no more than $3^d$ bad boxes. Hence, choosing $c_2 = 1/3^d$, we arrive at
\begin{multline*}
\P\big[ \Ll| \{ i \in A : B_i \text{ is a bad box}\} \Rr| \ge \eta\b  N \big] \\ \le \P\big[ \Ll| \{ i \in A : B_i \text{ is not balanced}\} \Rr| \ge c_2 \eta\b  N \big]	\le \exp\Ll(-c L\b^{\delta/2} N \Rr).
\end{multline*}
Multiplying this by our upper bound $3^{2dN}$ on the number of lattice animals, we have thus bounded the probability in the left-hand side of (\ref{e:animals0}) by
$$
\exp\Ll( -c L\b^{\delta/2} N + 2d\log(3)N \Rr).
$$
This proves that \eqref{e:animals} holds for $\beta$ small enough, and thus finishes the proof.
\end{proof}

\subsection{The cost of a good step}
\label{ss:costgoodstep}
In fact, the definition in \eqref{defbalanced} of a balanced box asks that important sites are ``nowhere missing'', but it may happen that they are in excess. Since we want to bound from below the sum of the $V$'s seen by the random walk, this should not be a problem. However, for the purpose of the proof, it will be convenient to extract a selection of important sites which will not be too numerous. 

Let $b'=B(x,r\b')$. By definition, $\1(b',l) = 1$ if and only if 
$$
\Ll| B(x,(1-\eps_1)r\b') \cap T_l \Rr| = \Ll| B(x,(1-\eps_1)r\b') \cap T \Rr| = 1.
$$
In this case, we let $x(b',l)$ be the unique element of these sets. Given a balanced box $B_i$, and $b \in \mcl{P}_i$, we know that the cardinality of the set 
$$
\mcl{I}'_{b,l} = \{x(b',l) \ | \ b' \in \mcl{P}'_b, \1(b',l) = 1\}
$$
is at least $(1-\eps) p_l |b|$. We choose in some arbitrary deterministic manner a subset $\mcl{I}_{b,l}$ of $\mcl{I}'_{b,l}$ whose cardinality lies in the interval $[(1-\eps) p_l |b|, p_l|b|]$ (this is possible for $\beta$ small since $p_l |b|$ tends to infinity as $\beta$ tends to $0$). We further define
$$
\mcl{I}_{i,l} = \bigcup_{b \in \mcl{P}_i} \mcl{I}_{b,l} \quad, \quad \mcl{I}_l = \bigcup_{\substack{i \in \Z^d \\ B_i \text{ balanced}}} \mcl{I}_{i,l}, \quad \text{and} \quad \mcl{I} = \bigcup_{l = 1}^\kappa \mcl{I}_l.
$$
Note that any two elements of $\mcl{I}$ are at a distance at least $2\eps_1 r\b'$ from one another (see Figure~\ref{f:boxes}).

We define
\begin{equation}
\label{deftau}
\tau = \inf\{k > 0 : S_k \notin D(S_0,R\b) \},
\end{equation}
and 
\begin{equation}
\label{defs}
s = \sum_{k=0}^{\tau-1} \beta V(S_k) \1_{\{ S_k \in {\mcl{I}} \cap D(S_0,R\b - r\b) \setminus D(S_0,\eps_1 r\b')\}}.
\end{equation}
Clearly, this last quantity is a lower bound on the ``cost'' of the first piece of the coarse-grained trajectory. The advantage of having dropped some important sites as we just did is this. Now, we will be able to show that if the walk starts within a good box, then with high probability the quantity $s$ is simply $0$, and the probability that two or more sites in ${\mcl{I}}$ actually contribute to the sum $s$ is negligible. Moreover, any important point contributing to $s$ is far enough from the boundary of $D(S_0,R\b)$ so that if visited, the returns to this site will most likely occur before exiting $D(S_0,R\b)$. See Figure~\ref{f:annulus} for an illustration of the construction.

\begin{figure}
\centering
\includegraphics[scale=0.7]{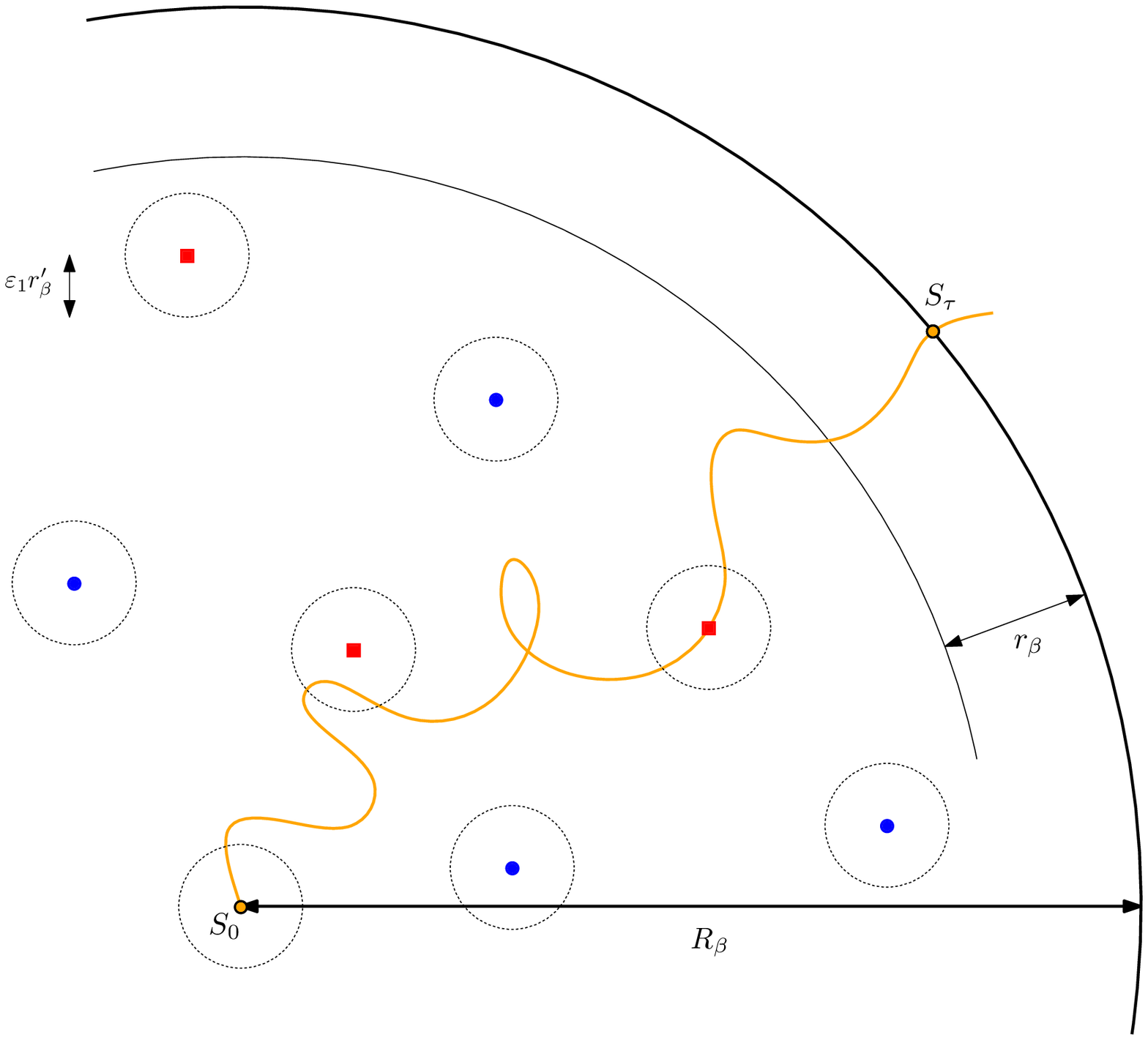}
\caption{
\small{
The trajectory of the random walk is represented by the orange path. Only important sites in $\mcl{I} \cap D(S_0,R\b - r\b) \setminus D(S_0,\eps_1 r\b')$ are depicted. There are two types of important sites, represented by blue circles and red squares. Dashed circles have radius $\eps_1 r\b'$. Recall that two sites in $\mcl{I}$ are at distance at least $2\eps_1 r\b'$ from one another. On the drawing, the unlikely event that the trajectory actually intersects the set $\mcl{I} \cap D(S_0,R\b - r\b) \setminus D(S_0,\eps_1 r\b')$ is realized.
}
}
\label{f:annulus}
\end{figure}

\begin{prop}
\label{p:laplaces}
If $\beta$ is small enough, then any $x \in \Z^d$ lying in a good box satisfies
\begin{equation}
\label{e:laplaces1}
\PP_x[s \neq 0] \le (1+\eps) q_d \ R\b^2 \ L\b^{-2}
\end{equation}
and
\begin{equation}
\label{e:laplaces2}
\EE_x\Ll[ e^{-s} \Rr] \le 1-(1-3\eps) R\b^2\ I\b,	
\end{equation}
where $q_d$ was defined in \eqref{defqd} and $I\b$ in \eqref{defIb}.
\end{prop}
The most important step is contained in this lemma.
\begin{lem}
\label{l:laplaces}
For $x \in \Z^d$ and $l \le \kappa$, define 
$$
W_{x,l} = \sum_{\substack{y \in \mcl{I}_l \cap D(x,R\b-r\b) \setminus D(x,\eps_1 r\b')}} \PP_x\Ll[S \text{ visits } y \text{ before exiting } D(S_0,R\b)\Rr].
$$
If $\beta$ is small enough, then any $x$ lying in a good box satisfies
\begin{equation}
\label{Wbounds}
(1-2\eps) q_d \ p_l \ R\b^2 \le W_{x,l} \le (1+\eps) q_d \ p_l \ R\b^2.
\end{equation}
\end{lem}
\begin{proof}[Proof of Lemma~\ref{l:laplaces}]
For greater clarity, we assume $x = 0$, and comment on the necessary modifications to cover arbitrary $x$ at the end of the proof.

Consider the set
$$
\mcl{B} = \Ll\{ b \in \bigcup_{i \in \Z^d} \mcl{P}_i : b \subset D(0,R\b-r\b) \text{ and } 0 \notin b \Rr\}.
$$
Elements of $\mcl{B}$ are of the form $B((2r\b+1)j,r\b)$, with $j \neq 0$. In particular, any point contained in $b \in \mcl{B}$ is at distance at least $r\b$ from the origin.

A lower bound on $W_{0,l}$ is
\begin{equation}
\label{lowerW}
\sum_{b \in \mcl{B}} \ \sum_{y \in \mcl{I}_{b,l}} \PP_0\Ll[S \text{ visits } y \text{ before exiting } D(0,R\b)\Rr].
\end{equation}
Since by assumption $0$ lies in a good box, any $b \in \mcl{B}$ belongs to a balanced box, and thus $\mcl{I}_{b,l}$ is well defined, and moreover, $|\mcl{I}_{b,l}| \in [(1-\eps) p_l |b|, p_l|b|]$.

We learn from \cite[Lemma~A.2 (147)]{bc} that
\begin{equation}
\label{estimgreen}
\PP_0\Ll[S \text{ visits } y \text{ before exiting } D(0,R\b)\Rr]  \ge {c_d} \ {q_d}\Ll(|y|^{2-d} - R\b^{2-d}\Rr) + O\Ll(|y|^{1-d}\Rr),
\end{equation}
where $|y|$ is the Euclidian norm of $y$,
$$
c_d = \frac{d}{2} \ \Gamma\Ll(\frac{d}{2} - 1\Rr) \pi^{-d/2},
$$
and $\Gamma$ is Euler's Gamma function. Let $z_f(b)$ denote the point of the box $b$ which is the furthest from the origin, and $z_c(b)$ be the closest (with respect to the Euclidian norm, and with some deterministic tie-breaking rule). Using the lower bound \eqref{estimgreen} in \eqref{lowerW}, we arrive at
\begin{equation}
\label{Wxsum}
W_{0,l} \ge (1-\eps) c_d \ q_d \ p_l \sum_{b \in \mcl{B}} |b| \Ll[|z_f(b)|^{2-d} - R\b^{2-d} + O\Ll(|z_c(b)|^{1-d}\Rr)\Rr].
\end{equation}
We first show that the error term is negligible, that is,
\begin{equation}
\label{Werror}
\sum_{b \in \mcl{B}} \ |b| \ |z_c(b)|^{1-d} = o\Ll(R\b^2\Rr).
\end{equation}
Let 
$$
\mcl{B}' = \bigcup_{b \in \mcl{B}} \bigcup_{z \in b} z + [-1/2,1/2)^d.
$$ The left-hand side of (\ref{Werror}) is equal to
$$
\int_{\mcl{B}'} \Ll(|z|^{1-d} + O(r\b |z|^{-d})\Rr) \ \d z.
$$
On one hand, we have
$$
\int_{\mcl{B}'} |z|^{1-d} \ \d z \sim \omega_d R\b ,
$$
where $\omega_d$ is the surface area of the unit sphere in $\R^d$. On the other hand,
$$
\int_{\mcl{B}'}  r\b |z|^{-d} \ \d z \sim \omega_d r\b \log\Ll(\frac{R\b}{r\b}\Rr),
$$
so that \eqref{Werror} holds indeed.

The same type of argument shows that
\begin{equation*}
\sum_{b \in \mcl{B}} \ |b| \ \Ll[|z_f(b)|^{2-d} - R\b^{2-d}\Rr] \sim  R\b^2 \int_{|z| \le 1} \Ll[ |z|^{2-d} - 1 \Rr] \ \d z ,
\end{equation*}
with
$$
\int_{|z| \le 1} \Ll[ |z|^{2-d} - 1 \Rr] \ \d z = \Ll(\frac{d}2 - 1\Rr) \frac{\omega_d}{d}.
$$
Recalling moreover that
$$
\frac{\omega_d}{d} = \frac{ \pi^{d/2}}{\Gamma\Ll( \frac{d}{2} + 1 \Rr)} = \frac{ \pi^{d/2}}{\frac{d}{2}(\frac{d}2 - 1) \Gamma\Ll( \frac{d}{2} - 1 \Rr)},
$$
we have proved that
$$
W_{0,l} \ge (1-\eps)  q_d \ p_l \ R\b^2 \ \Ll(1+o(1)\Rr),
$$
which implies the lower bound in (\ref{Wbounds}).

\medskip

We now turn to the upper bound. Recall that, as given by \cite[Theorem~1.5.4]{law}, there exists $C > 0$ such that
\begin{equation}
\label{boundgreen}
\PP_0[S \text{ visits } y] \le C \ |y|^{2-d}.	
\end{equation}
We will also use the more refined estimate that can be found in \cite[Lemma~A.2 (149)]{bc} stating that for $y \in D(0,R\b)$,
\begin{multline}
\label{boundgreen2}
\PP_0[S \text{ visits } y \text{ before exiting } D(0,R\b)] \\
 \le c_d \ q_d \Ll(|y|^{2-d} - R\b^{2-d} + O(|x|^{1-d})\Rr)\Ll( 1+O((R\b - |y|)^{2-d}) \Rr).
\end{multline}
We first treat the contribution of important sites lying in $b(0) \stackrel{\text{(def)}}{=} B(0,r\b)$. By definition, any important site that contributes to the sum must be at distance at least $\eps_1 r\b'$. Moreover, since $0$ is assumed to belong to a good box, one has $|\mcl{I}_{b(0),l}| \le p_l \ |b|$. Using also \eqref{boundgreen}, we can bound the contribution of sites lying in $b(0)$ by
$$
C \ p_l \ |b| \ (\eps_1 r\b')^{2-d}.
$$
It suffices to choose $\delta$ small enough to ensure that this quantity is $o(p_l R\b^2)$. More precisely, in order to have $|b| (r\b')^{2-d} \ll R\b^2$, one should impose 
\begin{equation}
\label{conddelta}
d\Ll(\frac{2}{d} + \delta\Rr) - (d-2)\Ll(\frac{2}{d} - \delta\Rr) < 2(1-\delta),	
\end{equation}
which is clearly true for any small enough $\delta$.

We now turn to the contribution of the important sites lying outside of $b(0)$. Let
$$
\mcl{B}'' = \Ll\{ b \in \bigcup_{i \in \Z^d} \mcl{P}_i : b \cap D(0,R\b) \neq \emptyset \text{ and } 0 \notin b \Rr\}.
$$
The contribution of the important sites belonging to a box of $\mcl{B}''$ is bounded from above by
\begin{equation}
\label{bound1}
\sum_{b \in \mcl{B}''} \ \sum_{y \in \mcl{I}_{b,l} \cap D(0,R\b-r\b)} \PP_0\Ll[S \text{ visits } y \text{ before exiting } D(0,R\b)\Rr].	
\end{equation}
We will now use the estimate given in \eqref{boundgreen2}. Note first that since the $y$'s considered in \eqref{bound1} are all in $D(0,R\b - r\b)$, the contribution of the error term $O((R\b - |y|)^{2-d})$ appearing in \eqref{boundgreen2} is negligible. Forgetting about this error term, and using also the fact that $|\mcl{I}_{b,l}| \le p_l \ |b|$, we get that the sum in \eqref{bound1} is smaller than
\begin{equation}
\label{comparethis}
c_d \ q_d \ p_l  \sum_{b \in \mcl{B}''}  |b| \ \Ll(|z_c(b)|^{2-d} - R\b^{2-d} + O(|z_c(b)|^{1-d})\Rr).	
\end{equation}
To conclude the analysis, one can proceed in the same way as for the lower bound (compare \eqref{comparethis} with \eqref{Wxsum}).

To finish the proof, we discuss how to adapt the above arguments to the case when $x$ is not the origin. In both arguments, we treated separately the box $b(0) = B(0,r\b)$. It has the convenient feature that any point outside of this box is at distance at least $r\b$ from the origin. For general $x$, the box $b(x)$ of the form $B((2r\b+1)j,r\b)$ ($j \in \Z^d$) containing $x$ need not have this feature. In this situation, one can consider separately the box $b(x)$ together with its neighbouring boxes on one hand, and all the other boxes on the other, and the above arguments still apply.
\end{proof}
\begin{proof}[Proof of Proposition~\ref{p:laplaces}]
In order to prove \eqref{e:laplaces1}, it suffices to observe that
\begin{eqnarray}
\label{comp1}
\PP_x[s \neq 0] & \le & \sum_{\substack{y \in \mcl{I} \cap D(x,R\b-r\b) \setminus D(x,\eps_1 r\b')}} \PP_x\Ll[S \text{ visits } y \text{ before exiting } D(x,R\b)\Rr] \notag \\
& \le & \sum_{l=1}^\kappa W_{x,l},
\end{eqnarray}
and to apply Lemma~\ref{l:laplaces}. 

Let $\mcl{E}^{(2)}$ be the event
$$
\begin{array}{l}
	S \text{ visits at least two distinct elements of } \\
\mcl{I} \cap D(S_0,R\b - r\b') \setminus D(S_0,\eps_1 r\b') \text{ before}\\
\text{exiting } D(S_0,R\b).
	\end{array}	
$$
From the above computation, one can also infer that 
\begin{equation}
\label{hittwo}
\PP_x\Ll[\mcl{E}^{(2)}\Rr] = o(R\b^2 I\b).
\end{equation}
Indeed, the computation in \eqref{comp1} shows that the probability to hit one element of $\mcl{I} \cap D(0,R\b - r\b') \setminus D(0,\eps_1 r\b')$ is bounded by a constant times $R\b^2 I\b$ (recall that $I\b \simeq L\b^{-2}$). Conditionally on hitting such a site, say $z$, one can apply the same reasoning to bound the probability to hit another trap by a constant times $R\b^2 I\b$. The key point is to observe that there is no other element of $\mcl{I}$ within distance $\eps_1 r\b'$ from $z$. The probability to hit another trap is thus bounded by
\begin{equation}
\label{e:eq1}
\sum_{\substack{y \in \mcl{I} \cap D(x,R\b-r\b) \setminus D(z,\eps_1 r\b')}} \PP_z\Ll[S \text{ visits } y \text{ before exiting } D(x,R\b)\Rr].	
\end{equation}
There is no harm in replacing $D(x,R\b)$ by $D(z,2R\b)$ in the above sum, which thus allows us to follow the proof of Lemma~\ref{l:laplaces} and obtain that the sum in \eqref{e:eq1} is $O(R\b^2 I\b)$. To sum up, we have thus shown that the probability in the left-hand side of \eqref{hittwo} is $O(R\b^4 I\b^{2}) = o(R\b^2 I\b)$.

Let us write $\mcl{E}_{l}$ for the event
$$
\begin{array}{l}
S \text{ visits an element of } \mcl{I}_l \cap D(S_0,R\b - r\b') \setminus D(S_0,\eps_1 r\b') \\
\text{before exiting } D(S_0,R\b),
\end{array}
$$
and $\mcl{E} = \bigcup_{k} \mcl{E}_l$. By the inclusion-exclusion principle, one has
\begin{multline*}
\PP_x\Ll[ \mcl{E}_l \Rr] \\
\ge W_{x,l} \ - \sum_{\substack{y < z \in \mcl{I}_l \\ y,z \in D(x,R\b-r\b) \setminus D(x,\eps_1 r\b')}} \PP_x\Ll[S \text{ visits } y \text{ and } z \text{ before exiting } D(x,R\b)\Rr].
\end{multline*}
The sum in the right-hand side above is smaller than $\PP_x\Ll[\mcl{E}^{(2)}\Rr]$, and is thus $o(R\b^2 I\b)$. Using Lemma~\ref{l:laplaces}, we arrive at
\begin{equation}
\label{boundonEk}
\PP_x\Ll[ \mcl{E}_l \Rr] \ge (1-2\eps) q_d \ p_l \ R\b^2(1+o(1)).
\end{equation}
%Let us write $\mcl{E}_{k}$ for the event
%$$
%\begin{array}{l}
%	S \text{ visits exactly one element of } \mcl{I}_k \cap D(S_0,R\b - r\b') \setminus D(S_0,\eps_1 r\b') \\
%\text{before exiting } D(S_0,R\b).
%\end{array}	
%$$
%One also has
%$$
%\PP_x\Ll[\mcl{E}_{k}\Rr] \ge P_{x,k} - \PP_x\Ll[\mcl{E}^{(2)}\Rr] = (1-2\eps) p_k \ R\b^2(1+o(1)).
%$$
Similarly,
\begin{equation}
\label{boundonE}
\PP_x\Ll[ \mcl{E} \Rr]	\ge (1-2\eps) q_d \ L\b^{-2}  \ R\b^2(1+o(1)).
\end{equation}
We can now decompose the expectation under study the following way
\begin{equation}
\label{decompes}
\EE_x[e^{-s}] \le \PP_x[s=0] + \sum_{l=1}^\kappa \EE_x\Ll[e^{-s}, \mcl{E}_{l}\Rr].	
\end{equation}
On one hand, we have
\begin{equation}
\label{expterm1}
\PP_x[s = 0] = 1- \PP_x[s \neq 0] \stackrel{\text{\eqref{boundonE}}}{\le} 1-(1-3\eps) q_d \ L\b^{-2}  \ R\b^2.
\end{equation}
On the other hand, conditionally on hitting a point $y \in \mcl{I}_k \cap D(0,R\b-r\b)$, the walk does a geometrical number of returns to $y$ with a return probability equal to
$$
\PP_y[S \text{ returns to } y \text{ before exiting } D(0,R\b)].
$$
This probability tends to $(1-q_d)$ uniformly over $y \in D(0,R\b-r\b)$. As a consequence, up to a negligible error, $\EE_x\Ll[e^{-s} \ | \ \mcl{E}_{l}\Rr]$ is bounded by
\begin{equation}
\label{expterm21}
\sum_{k = 1}^{+ \infty} e^{-k a_l} (1-q_d)^{k-1} q_d = \frac{q_d \ e^{-a_l}}{1 - (1-q_d) e^{-a_l}} .
\end{equation}
Combining \eqref{expterm1}, \eqref{expterm21} and \eqref{boundonEk}, we obtain
$$
\EE_x[e^{-s}] \le 1-(1-3\eps) q_d \  R\b^2 \ \sum_{l=1}^\kappa p_l \Ll( 1- \frac{q_d \ e^{-a_l}}{1 - (1-q_d) e^{-a_l}} \Rr),
$$
which is precisely the bound \eqref{e:laplaces2}.
\end{proof}

\subsection{The cost of going fast to the hyperplane}
%The following gives a precise meaning to the informal statement in \eqref{met1}. (* useless ! *)
%\begin{prop}
%\label{met1++}
%Let $v > 0$. There exists $J(v)$ such that
%\begin{equation}
%\label{e:met1++}
%\lim_{n \to + \infty} -\frac{1}{n} \ \log \PP_0\Ll[T_n(\ell) \le n/v \Rr] = J(v),
%\end{equation}
%with $J(v) \sim {d v}/{2}$ as $v$ tends to $0$.
%\end{prop}
%\begin{proof}
%Note first that
%\begin{equation}
%\label{e:met1++1}
%\PP_0[S_{\lfloor n/v \rfloor} \cdot \ell \ge n] \le \PP_0\Ll[T_n(\ell) \le n/v \Rr] \le \sum_{k = 0}^{n/v} \PP_0[S_{k} \cdot \ell \ge n].
%\end{equation}
%Letting 
%$$
%\Psi(x) = \sup_{\lambda} \Ll( \lambda x   - \log \EE_0\Ll[e^{\lambda (S_1 \cdot \ell)}\Rr]\Rr),
%$$
%classical large deviation results yield that
%\begin{equation}
%\label{largedev1}
%\lim_{n \to +\infty} -\frac{1}{n} \ \log \PP_0[S_{\lfloor n/v \rfloor} \cdot \ell \ge n] = \frac{\Psi(v)}{v},	
%\end{equation}
%and moreover, for any $k$ and $n$,
%$$
%\PP_0[S_{k} \cdot \ell \ge n] \le \exp\Ll( - k \Psi(n/k) \Rr).
%$$
%Since $\Psi$ is convex and $\Psi(0) = 0$, one has, for any $k \le n/v$,
%$$
%\frac{\Psi(n/k)}{n/k} \ge \frac{\Psi(v)}{v},
%$$
%and as a consequence,
%$$
%\PP_0[S_{k} \cdot \ell \ge n] \le \exp\Ll( - n \ \frac{\Psi(v)}{v} \Rr).
%$$
%This enables to control the right-hand side of \eqref{e:met1++1}, while equation \eqref{largedev1} handles its l.h.s. We thus obtain equation \eqref{e:met1++} with $J(v) = \Psi(v)/v$. The fact that $J(v) \sim dv/2$ follows from observing that
%$$
%\EE_0\Ll[e^{\lambda (S_1 \cdot \ell)}\Rr] - 1 \sim \frac{\lambda^2}{2d} \quad (\lambda \to 0).
%$$
%\end{proof}

We will need some information regarding the displacements at the coarse-grained scale. Let $X$ be the position of the particle when exiting the ball $D(S_0,R\b)$, that is, $X = S_{j_1} - S_0$, where $j_1$ is the exit time from $D(S_0,R\b)$, defined in \eqref{defjn}.
\begin{prop}
\label{unitX}
Let $\lambda\b$ be such that $\lambda\b \ll R\b^{-1}$. As $\beta$ tends to $0$, one has
$$
\EE_x\Ll[ \exp\Ll(\lambda\b (X \cdot \ell)\Rr) \Rr] \le 1+\frac{\lambda\b^2 R\b^2}{2d} \ (1 + o(1)).
$$	
\end{prop}
\begin{proof}
Observe that, since $X \cdot \ell \le 2 R\b$, one has
$$
\EE_x\Ll[ \exp\Ll(\lambda\b (X \cdot \ell)\Rr) \Rr] = 1 +  \frac{\lambda\b^2}{2} \ \EE_x\Ll[(X \cdot \ell)^2\Rr] + o(\lambda\b^2 R\b^2).
$$
The functional central limit theorem ensures that the distribution of $X/R\b$ approaches the uniform distribution over the unit sphere as $\beta$ tends to $0$. Writing $\sigma$ for the latter distribution, we need to show that
$$
\int (x \cdot \ell) \ \d \sigma(x) = \frac{1}{d}.
$$
In order to do so, one can complete $\ell$ into an orthonormal basis $\ell = \ell_1, \ldots, \ell_d$, and observe that, by symmetry,
$$
d \ \int (x \cdot \ell) \ \d \sigma(x) = \sum_{i = 1}^d \int (x \cdot \ell_i) \ \d \sigma(x) = 1.
$$
\end{proof}

\subsection{Asymptotic independence of $s$ and $X$}
\begin{prop}
\label{unitsX}
Let $\lambda\b$ be such that $\lambda\b \ll R\b^{-1}$. For $\beta$ small enough and any $x$ lying in a good box,
$$
\EE_x\Ll[ \exp\Ll(-s + \lambda\b \ X \cdot \ell \Rr) \Rr] \le 1 - (1-4\eps)R\b^2 \ I\b +\frac{\lambda\b^2 R\b^2}{2d} \ (1 + \eps).
$$	
\end{prop}
\begin{proof}
We have the following decomposition
\begin{multline}
\label{unitsXdecomp}
\EE_x\Ll[ \exp\Ll(-s + \lambda\b \  X \cdot \ell\Rr) \Rr] = \\
\EE_x\Ll[ \exp\Ll(\lambda\b\  X \cdot \ell\Rr) \ \1_{s = 0} \Rr] + \EE_x\Ll[ \exp\Ll(-s + \lambda\b \ X \cdot \ell\Rr)  \ \1_{s \neq 0}\Rr].	
\end{multline}
Let us first evaluate the first term in the right-hand side above. 
\begin{eqnarray*}
\EE_x\Ll[ \exp\Ll(\lambda\b \ X \cdot \ell\Rr) \ \1_{s = 0} \Rr] & = & \EE_x\Ll[ \exp\Ll(\lambda\b \ X \cdot \ell\Rr) \Rr] - \EE_x\Ll[ \exp\Ll(\lambda\b \ X \cdot \ell\Rr) \ \1_{s \neq 0} \Rr] \\
& \le & 1+\frac{\lambda\b^2 R\b^2}{2d}(1+o(1)) - \PP_x[s \neq 0]\Ll( 1+O(\lambda\b R\b) \Rr),
\end{eqnarray*}
where we used Proposition~\ref{unitX}. For the second term in the right-hand side of \eqref{unitsXdecomp}, we have
$$
\EE_x\Ll[ \exp\Ll(-s + \lambda\b \ X \cdot \ell\Rr)  \ \1_{s \neq 0}\Rr] = \Ll(1 + O(\lambda\b R\b)\Rr) \EE_x\Ll[ e^{-s} \ \1_{s \neq 0}\Rr].
$$
Moreover,
$$
\EE_x\Ll[ e^{-s} \ \1_{s \neq 0}\Rr] = \EE_x\Ll[ e^{-s}\Rr] - \PP_x[s = 0].
$$
We learn from Proposition~\ref{p:laplaces} that
$$
\EE_x\Ll[ e^{-s}\Rr] \le 1-(1-3\eps)R\b^2 \ I\b.
$$
To sum up, we have shown that
\begin{multline*}
\EE_x\Ll[ \exp\Ll(\lambda\b \ X \cdot \ell\Rr) \ \1_{s = 0} \Rr] \le 1+\frac{\lambda\b^2 R\b^2}{2d}(1+o(1)) - \PP_x[s \neq 0]\Ll( 1+O(\lambda\b R\b) \Rr) \\
+ \Ll(1 + O(\lambda\b R\b)\Rr) \Ll(\PP_x[s \neq 0]  -(1-3\eps)R\b^2 \ I\b  \Rr).
\end{multline*}
Since Proposition~\ref{p:laplaces} ensures that $\PP_x[s \neq 0] \le 2 R\b^2 L\b^{-2}$, and since $L\b^{-2} \simeq I\b$, the result follows.
\end{proof}

\subsection{Discarding slow motions}
\label{ss:slowmotion}
We recall that $S_0, S_{j_1},S _{j_2}, \ldots$ denote the successive steps of the trajectory coarse-grained at scale $R\b$. Let $k_n$ be defined by
\begin{equation}
\label{defkn}
k_n = \inf  \Ll\{ k : D(S_{j_k}, R\b) \cap H_n(\ell) \neq \emptyset \Rr\},
\end{equation}
where 
$$
H_n(\ell) = \left\{ x \in \Z^d : x \cdot \ell \ge n   \right\}
$$
is the half-space not containing $0$ delimited by the hyperplane orthogonal to $\ell$ and at distance $n$ from the origin. By the definition of $T_n(\ell)$, we have
$$
T_n(\ell) \ge j_{k_n}.
$$

We first want to discard overly slow behaviours. Out of the sequence of $k_n$ coarse-graining instants $j_0, \ldots, j_{k_n-1}$, we extract a maximal subsequence $j_0', \ldots, j_{K_n-1}'$ such that for any any $0 \le k < K_n$, $S_{j_k'}$ lies in a good box. For $0 \le k < K_n$, we define
\begin{equation}
\label{defsXk}
s_k = s \circ \Theta_{j_k'} \quad \text{and} \quad X_k = X \circ \Theta_{j_k'},	
\end{equation}
where, for $t \in \N$, $\Theta_t$ denotes the time translation by $t$ on the space of trajectories, that is, $(\Theta_t S)_i = S_{t+i}$, and we recall that $s$ was defined in \eqref{defs} and $X = S_{j_1} - S_0$. Note that the $j_k'$ are stopping times (under $\PP_0$ for every fixed environment, and with respect to the natural filtration associated to $S$). The number $K_n$ counts how many coarse-graining instants prior to $j_{k_n}$ are such that the walk at these instants lies in a good box.

\begin{prop}
\label{noslowmotion} 
For $\beta$ small enough, any $c_3 > 0$ and almost every environment, one has
\begin{equation}
\label{e:noslowmotion}
\EE_0 \Ll[ \exp\Ll( -\sum_{k = 0}^{T_n(\ell)-1} \beta V(S_k)
% \ \1_{\beta V(S_k) \ge a}
 \Rr), K_n > \frac{n}{c_3 R\b^2 \sqrt{I\b}} \Rr] 
 \le \exp\Ll(-\frac{\sqrt{I\b}}{2 c_3} \ n\Rr).
\end{equation}
\end{prop}
\begin{rem}
For small $\beta$, the walk makes on average $R\b^2$ steps in each coarse-grained unit of time. Roughly speaking, the event $K_n > n/R\b^2 v\b$ corresponds to asking $T_n(\ell)$ to be in the interval $(n/v\b,+\infty)$. We will use this proposition with $c_3$ large, so as to discard the possibility that $K_n$ be too large.
\end{rem}
\begin{proof}
Note that with probability one, the sequence $(s_k)$ can be defined for any $k \in \N$ (that is, we may as well not stop the sequence at $K_n$). The left-hand side of \eqref{e:noslowmotion} is smaller than
\begin{equation*}
\EE_0\Ll[ \exp \Ll( -\sum_{k = 0}^{K_n-1} s_k \Rr), K_n \ge\frac{n}{c_3 R\b^2 \sqrt{I\b}} \Rr] 
 \le \EE_0\Ll[ \exp \Ll( -\sum_{k = 0}^{n/(c_3 R\b^2 \sqrt{I\b})-1} s_k \Rr) \Rr].
\end{equation*}
Using the Markov property and Proposition~\ref{p:laplaces}, we obtain that for $\beta$ small enough, the latter is smaller than
\begin{eqnarray*}
\Ll( 1-(1-3\eps) R\b^2 I\b \Rr)^{n/(c_3 R\b^2 \sqrt{I\b})} & = & \exp\Ll( \frac{n}{c_3 R\b^2 \sqrt{I\b}} \ \log\Ll( 1-(1-3\eps) R\b^2 I\b \Rr) \Rr) \\
& \le & \exp\Ll( - \frac{\sqrt{I\b}}{2c_3} n \Rr),
\end{eqnarray*}
where in the last step, we used the fact that $-\log(1-x) \sim x$ as $x$ tends to $0$, and that $(1-3\eps) > 1/2$.
\end{proof}

\subsection{Path surgery}
\label{ss:pathsurge}
Before discussing the costs associated to a range of speeds that should contain the optimal speed, we introduce a ``surgery'' on the trajectory of the random walk, which consists in erasing certain annoying loops.

We first introduce general notations. Given $0 = g_0 \le h_1 \le g_1 \le h_2 \le \cdots \le g_j \le h_{j+1} = K$, we write
\begin{equation}
\label{defsurg}
[g,h] = \{k \text{ s.t. } \exists i : g_i \le k < h_{i+1} \}.
\end{equation}
We call $[g,h]$ a \emph{surgery} of $\{0,\ldots,K-1\}$ with \emph{at most} $j$ \emph{cuts}, where the cuts are the sets
$$
\{k \text{ s.t. }  h_i \le k < g_{i} \}, \quad (1 \le i \le j)
$$
whose union is the complement of $[g,h]$ in $\{0,\ldots,K-1\}$. Note that, since we allow the possibility that $h_i = g_i$ or $g_i = h_{i+1}$, it is not possible to recover $j$ if given $[g,h]$, hence the phrasing ``\emph{at most} $j$ cuts''.

\medskip

Let us now discuss why some surgery is needed, and how we choose the surgery $[g,h]$.
For simplicity, we write $Y_k = S_{j_k}$ for the coarse-grained random walk. By the definition of $k_n$ (see \eqref{defkn}), we have $Y_{k_n-1} \cdot \ell \ge n - 2R\b$. Based on our previous work, we should be able to argue that there are only few different bad boxes visited by $Y$, and we would like to conclude that
\begin{equation}
\label{pasvrai}
\sum_{k = 0}^{K_n - 1} X_k \cdot \ell \ge (1-\eps)n
\end{equation}
(recall that $\sum_{k = 0}^{K_n - 1} X_k$ is the sum of the increments of the coarse-grained walk that start within a good box). In other words, we would like to say that the sum of increments that start within a bad box gives a negligible contribution in the direction of $\ell$. This may however fail to hold, even if the number of bad boxes visited is really small, since it may happen for instance that the walk visits many times the same bad box and every time makes a jump in the direction of~$\ell$.

Instead of trying to control the trajectory of the walk on bad boxes (which would be a very delicate matter), we introduce a surgery on the path. Each time a new bad box is discovered, we remove the piece of trajectory between the first and last visits to the bad box. We may call the piece of trajectory which is removed a \emph{loop} since, although it possibly does not intersect itself, the starting and ending points are in the same box. Once these pieces of trajectory have been removed, the remaining increments should satisfy an inequality like \eqref{pasvrai}. 

More precisely, we define
$$
H_1 = \inf\{k < k_n : Y_k \text{ lies in a bad box}\},
$$
$$
G_1 = 1+\sup\{k < k_n : Y_k \text{ lies in the same bad box as } Y_{H_1}\},
$$
and then, recursively, as long as $H_j$ is well defined,
$$
H_{j+1} = \inf\{k : G_j \le k < k_n \text{ and } Y_k \text{ lies in a bad box}\},
$$
$$
G_{j+1} = 1+\sup\{k < k_n : Y_k \text{ lies in the same bad box as } Y_{H_{j+1}}\}.
$$
We let $J$ be the largest $j$ such that $H_j$ is well defined, and set $H_{J+1} = k_n$, $G_0 = 0$ (see Figure~\ref{f:surgery} for an illustration).

\begin{figure}
\centering
\includegraphics[scale=0.63]{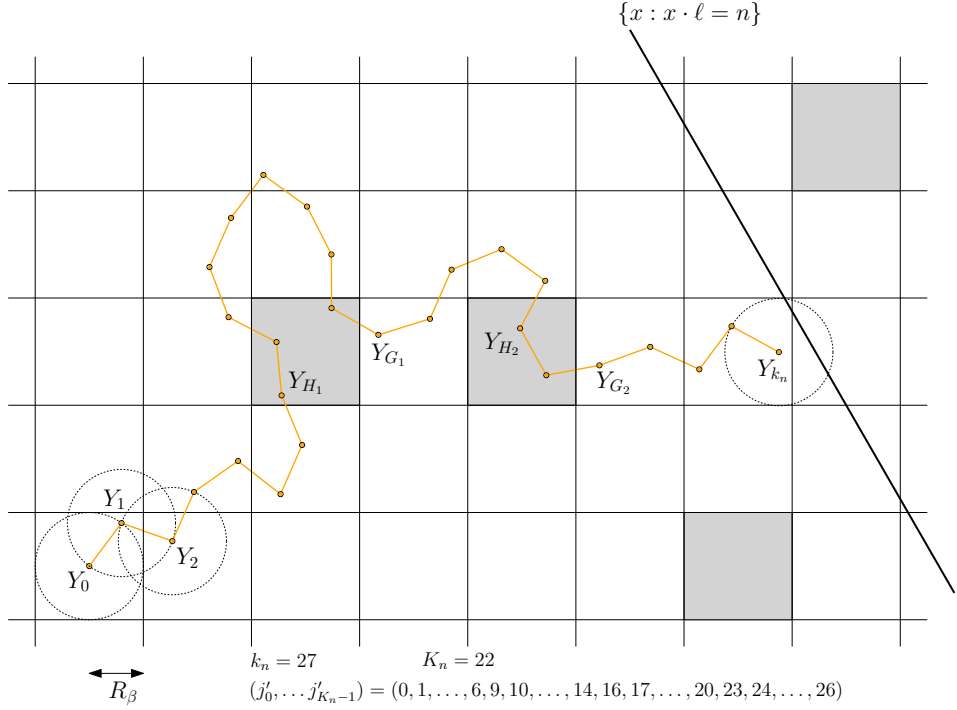}
\caption{
\small{
Orange dots represent the coarse-grained trajectory $Y_0, Y_1, \ldots,$ up to the moment when the coarse-grained trajectory is within distance $R_\beta$ from the target hyperplane (depicted as a thick line). The dots are linked by an orange line for better visualisation. The grey squares are bad boxes (the definition of bad boxes makes it impossible to have an isolated bad box, but since this is not relevant for the definition of the surgery, we did not take this restriction into consideration on this picture). Note that the boxes visited by the coarse-grained trajectory form a lattice animal.
}
}
\label{f:surgery}
\end{figure}

For any $j$, one has
$$
\sum_{k=H_j}^{G_j-1} (Y_{k+1}-Y_k) \cdot \ell = (Y_{G_j} - Y_{H_j}) \cdot \ell \le 4 \sqrt{d} R\b,
$$
since $Y_{G_j-1}$ and $Y_{H_j}$ lie in the same box (this is the loop, that we cut out). Hence,
\begin{equation}
\label{lacestvrai}
Y_{k_n} \cdot \ell \le 4 \sqrt{d} \ J \ R\b + \sum_{j=1}^J \sum_{k = G_j}^{H_{j+1}-1} (Y_{k+1} - Y_k) \cdot \ell,
\end{equation}
and we may rewrite the last double sum as
\begin{equation}
\label{cutonk}
\sum_{k \in [G,H]} (Y_{k+1} - Y_k) \cdot \ell.
\end{equation}
Clearly, $J$ is smaller than the number of different bad boxes visited. Moreover, by definition, if $k$ is such that $G_j \le k < H_{j+1}$ for some $j$ (that is, if $k \in [G,H]$), then $Y_k$ lies in a good box. In other words, the summands in \eqref{cutonk} form a subsequence of the summands in the left-hand side of \eqref{pasvrai}.

Considering \eqref{lacestvrai} and the fact that $Y_{k_n} \cdot \ell \ge n - R\b$, we obtain a lower bound on the sum in \eqref{cutonk}:
\begin{equation}
\label{letsclarify}
\sum_{k \in [G,H]} (Y_{k+1} - Y_k) \cdot \ell \ge n - R\b - 4 \sqrt{d} J R\b,
\end{equation}
which we be useful as soon as we have a good upper bound on $J$, the number of different bad boxes visited.

The set $[G,H]$ is a surgery on the set $\{0,\ldots, k_n-1\}$ that indexes the successive jumps of the coarse-grained walk, and moreover, we recall that any $k \in [G,H]$ is such that $Y_k$ lies in a good box. We now transform it into a surgery $[G',H']$ of the set $\{0,\ldots, K_n-1\}$ that indexes the successive jumps of the coarse-grained walk that start in a good box (call this a ``good increment''). More precisely, for $1 \le i \le J$, define $H_i'$ to be the index of the last good increment occurring before $H_i$, $G'_i$ to be the index of the first good increment occurring at or after $G_i$, together with $G'_0 = 0$ and $H'_{J+1} = K_n-1$. With this notation, we have
\begin{equation}
\label{reindex}
\sum_{k \in [G,H]} (Y_{k+1} - Y_k) \cdot \ell = \sum_{k \in [G',H']} X_k \cdot \ell,	
\end{equation}
where we recall that $X_k$ was defined in \eqref{defsXk}. The passage from the right-hand side to the left-hand side in \eqref{reindex} is purely a re-indexation of the terms, each of them appearing in both sides and in the same order. Note that $[G',H']$ has at most $J$ cuts, and $J$ is bounded by the number of distinct bad boxes visited by the walk.

\subsection{On the number of possible surgeries}
A negative aspect of the surgery is that $[G',H']$ depend on the full trajectory of the walk up to hitting the hyperplane. To overcome this problem, we will use a union bound on all reasonable surgeries, and then examine each deterministic surgery $[g,h]$ separately. We thus need a bound on the number of these surgeries. The bound we need can grow exponentially with~$n$, but the prefactor must be small compared to $\sqrt{I\b}$.

We start with a combinatorial lemma.

\begin{lem}
\label{combinat}
Let $c,c' >0$, and $N_n$ be the total number of possible surgeries of the set $\{0,\ldots,c n\}$ by at most $c' n$ cuts. We have
$$
N_n \le \exp\Ll(\Ll[2c'\log\Ll(1+\frac{c}{2c'}\Rr) + 2c' \Rr]n \Rr).
$$
\end{lem}
\begin{proof}
The surgery in \eqref{defsurg} is defined by giving $0 = g_0 \le h_1 \le g_1 \le h_2 \le  \cdots \le g_j \le h_{j+1}$, and in our present setting, we impose $h_{j+1} = cn$ and $j \le c'n$. Hence, in order to define such a surgery, it is sufficient to give oneself an increasing sequence (in the wide sense) of $2c' n$ elements in $\{0,\ldots, c n \}$. 

We now proceed to count these objects. Consider such a sequence $(u_i)_{1 \le i \le 2c' n}$. We think of $0,\ldots, c n$ as a string of characters, and for each $i$, we insert a character~$\star$ in this string just before the value taken by $u_i$ (for instance, $0 \ 1 \ \star \ 2 \ \star \ \star \ 3$ is the string obtained from the sequence $2,3,3$). One can see that there is a bijective correspondence between increasing sequences and the set of positions of the character $\star$ within the string, provided we do not allow a $\star$ as the last character. The number of increasing sequences of length $2c' n$ in $\{0,\ldots, c n \}$ is thus equal to
$$
\binom{c n + 2c'n}{2c'n} \le \frac{(cn + 2c'n)^{2c'n}}{(2c'n)!}.
$$
The result is then obtained using the fact that $\log(n!) \ge n \log(n) - n$ (and the latter can be checked by induction on $n$). 
\end{proof}

Recalling that the space is partitioned into the family of boxes $(B_i)_{i \in \Z^d}$, we let 
$$
{A}_n = \Ll\{i \in \Z^d : \exists k < k_n \text{ s.t. } S_{j_k} \in B_i \Rr\}
$$
be the set of indices of the boxes visited by the coarse-grained trajectory. Since the boxes are of size $R\b$, the set ${A}_n$ is a lattice animal when the walk is started at the origin (see Figure~\ref{f:surgery}). Moreover, the box containing $S_{j_{k_n-1}}$ is within distance $2R\b$ from the hyperplane, which is itself at distance $n$ from the origin. It thus follows that, for any $c_4 > 2 \sqrt{d}$ and $\beta$ small enough,
\begin{equation}
\label{boundonAn}
\Ll| {A}_n \Rr| \ge \frac{n}{c_4 R\b}.
\end{equation}
Let $\mcl{A}_n$ be the event defined by
\begin{equation}
\label{defAn}
\Ll| \{ i \in {A}_n : B_i \text{ is a bad box}\} \Rr| < \eta\b \ {A}_n,
\end{equation}
where we recall that $\eta\b = L\b^{-5\delta/2}$. We get from Proposition~\ref{animals} and inequality \eqref{boundonAn} that 
\begin{equation}
\label{boundforAn}
\P[\mcl{A}_n^c] \le e^{-n/c_4 R\b}.
\end{equation}
Discarding events with asymptotic cost $(c_4 R\b)^{-1}$, we can focus our attention on the environments for which $\mcl{A}_n$ holds. 

\medskip
We now fix
\begin{equation}
\label{defc1}
c_3 = \eps/2.
\end{equation}
This way, the lower bound for the cost obtained in Proposition~\ref{noslowmotion} is $\eps^{-1} \sqrt{I\b}$, which is much larger than the cost we target to obtain in the end, that is, $\sqrt{2d\ I\b}$ (to this end, we could as well choose a larger $c_3$, but having a separation of the order of $\eps^{-1}$ will prove useful in Section~\ref{s:intermediate2}). In other words, with this choice of $c_3$, we can assume from now on that the number of good steps made by the random walk satisfies
\begin{equation}
\label{boundonKn}
K_n \le \frac{n}{c_3 R\b^2 \sqrt{I\b}}.
\end{equation}

Since each good box has to be visited by the coarse-grained trajectory at least once, when condition \eqref{boundonKn} holds, we have
\begin{equation}
\label{boundongood}
\Ll| \{ i \in {A}_n : B_i \text{ is a good box}\} \Rr| \le \frac{n}{c_3 R\b^2 \sqrt{I\b}}.
\end{equation}

Using this together with \eqref{defAn}, we obtain that
\begin{equation}
\label{boundonbad}
\Ll| \{ i \in {A}_n : B_i \text{ is a bad box}\} \Rr| < \frac{\eta\b}{1-\eta\b} \frac{n}{c_3 R\b^2 \sqrt{I\b}},
\end{equation}
where 
\begin{equation}
\label{evalc'}
\frac{\eta\b}{1-\eta\b} \frac{1}{c_3 R\b^2 \sqrt{I\b}} \simeq {L\b^{-\delta/2}}{\sqrt{I\b}}.
\end{equation}

\begin{prop}
\label{controlcuts}
There exists $c_5 > 0$ such that for all $\beta$ small enough, when $\mcl{A}_n$ and \eqref{boundonKn} are both satisfied, the number of possible values of $[G',H']$ (that is, the cardinality of the set of surgeries having non-zero probability) is bounded by
$$
\exp\Ll( \Ll[ c_5 {L\b^{-\delta/4}}{\sqrt{I\b}} \Rr] n \Rr),
$$
and moreover, if $\eps n \ge 2 R_\beta$, then
$$
\sum_{k \in [G',H']} X_k \cdot \ell \ge (1-\eps)n.
$$
\end{prop}
\begin{proof}
For the first part, considering that the number of bad boxes visited by the coarse-grained walk is an upper bound for the number $J$ of cuts in the surgery $[G',H']$, together with \eqref{boundonbad}, we can apply Lemma~\ref{combinat} with $c = 1/(c_3 R\b^2 \sqrt{I\b})$ and $c'$ equal to the left-hand side of \eqref{evalc'}. The conclusion follows since $\log(1+\frac{c}{2c'})$ is then the logarithm of some power of $L\b$, so is smaller than $L\b^{\delta/4}$ when $\beta$ is small enough. The second part also follows from the bound on $J$, together with \eqref{letsclarify} and \eqref{reindex}.
\end{proof}

\subsection{Speeds and their costs}
We can now give precise estimates for the cost of speeds of the order of $\sqrt{I\b}$ or higher. We write $|[G',H']|$ for the cardinality of the set $[G',H']$, and for $0 \le v\b < v\b'$, we write $\mcl{E}_n(v\b,v\b')$ for the conjunction of the events
\begin{equation}
\label{defEn}
\begin{array}{c}
\displaystyle{\frac{n}{R\b^2 v\b'} \le |[G',H']| < \frac{n}{R\b^2 v\b},} \\
\displaystyle{\mcl{A}_n \text{ and  \eqref{boundonKn} hold},}
\end{array}
\end{equation}
where we recall that $\mcl{A}_n$ is the event defined in \eqref{defAn}. We write $\mcl{E}_n^c$ for the complement of the event $\mcl{E}_n(0,+\infty)$, that is, for the event when either $\mcl{A}_n$ or \eqref{boundonKn} fails to hold.

Recalling that on average, the walk makes $R\b^2$ steps during each coarse-grained displacement, and forgetting about the path surgery, one can interpret the event $\mcl{E}_n(v\b, v\b')$ as asking the random walk to move with a speed contained in the interval $[v\b, v\b')$. We further define
\begin{equation}
\label{defenabv}
e_{\beta,n}(v\b,v\b') = \E \EE_0 \Ll[ \exp\Ll( -\sum_{k = 0}^{T_n(\ell)-1} \beta V(S_k)  \Rr), \mcl{E}_n(v\b,v\b') \Rr],
\end{equation}
\begin{equation}
\label{defenc}
e_{\beta,n}^c = \E \EE_0 \Ll[ \exp\Ll( -\sum_{k = 0}^{T_n(\ell)-1} \beta V(S_k) %\ \1_{\beta V(S_k) \in [a,b)} 
\Rr), \mcl{E}_n^c \Rr].
\end{equation}
\begin{prop}
\label{costofspeed}
\begin{enumerate}
\item
For $\beta$ small enough, one has
\begin{equation}
\label{eq:costcomplement}
e_{\beta,n}^c \le 2 \exp\Ll(- \eps^{-1} \sqrt{I\b}\ n\Rr).
\end{equation}
\item Let $c, c'> 0$. If $v\b < v\b'$ satisfy $v\b \ll R\b^{-1}$ and either $v\b \ge c \sqrt{I\b}$ or $v\b' \le c' \sqrt{I\b}$, then for any small enough $\beta$, one has
\begin{equation}
\label{eq:costofspeed}
e_{\beta,n}(v\b,v\b') \le 2 \exp\Ll(-(1-5\eps)\Ll[\frac{dv\b}{2} + \frac{I\b}{v\b'}\Rr] \ n \Rr),
\end{equation}
\end{enumerate}
\end{prop}
\begin{rem}
Thinking about $\eps \to 0$, $a \to 0$, this can be seen as a rigorous form of the informal statement \eqref{met2} given in the introduction.
\end{rem}
\begin{proof}
For part (1), we saw in \eqref{boundforAn} that $\E[\mcl{A}_n^c] \le e^{-n/(c_4 R\b)}$. Since $R\b^{-1} \gg \sqrt{I\b}$, this term is smaller than $\exp\Ll(- \eps^{-1} \sqrt{I\b}\ n\Rr)$ for small enough $\beta$. The claim is then a consequence of Proposition~\ref{noslowmotion} and of our choice of $c_3$ (see \eqref{defc1}).

\medskip

Concerning part (2), the claim holds if $\beta$ is small enough and $\eps n < 2 R_\beta$. From now on, we consider only the case when $\eps n \ge 2 R_\beta$. An upper bound on $e_{\beta,n}(v\b,v\b')$ is given by
$$
\E \EE_0 \Ll[ \exp\Ll( -\sum_{k \in [G',H']} s_k \Rr), \mcl{E}_n(v\b,v\b') \Rr].
$$
%We need to prove that, as $n$ tends to infinity,
%$$
%\E \EE_0 \Ll[ \exp\Ll( -\sum_{k \in [G',H']} s_k \Rr), \mcl{E}_n(v\b,v\b') \Rr] \le \exp\Ll( -(1-5\eps)\Ll[\frac{dv\b}{2} + \frac{I\b}{v\b'} + o(1) \Rr] n\Rr).
%$$
Note that it follows from the assumptions that 
\begin{equation}
\label{boundbelow}
\frac{dv\b}{2} + \frac{I\b}{v\b'} \ge \td{c} \sqrt{I\b}
\end{equation}
for some fixed $\td{c}$. We let $\mcl{S}_n$ be the set of surgeries $[g,h]$ such that
$$
\P\PP_0\Ll[ [G',H'] = [g,h], \mcl{E}_n(v\b,v\b')\Rr] > 0.
$$
In view of \eqref{boundbelow}, Proposition~\ref{controlcuts} ensures that for $\beta$ small enough, the cardinality of the set $\mcl{S}_n$ is smaller than
$$
\exp\Ll(\eps \ \Ll[\frac{dv\b}{2} + \frac{I\b}{v\b'}\Rr] \ n \Rr).
$$
It moreover guarantees that on the conjunction of the events $[G',H'] = [g,h]$ and $\mcl{E}_n(v\b,v\b')$, one has
$$
\sum_{k \in [g,h]} X_k \cdot \ell \ge (1-\eps)n.
$$
It thus suffices to show that, for any sequence of cuts $[g_n,h_n] \in \mcl{S}_n$, one has
\begin{multline}
\label{goalofspeed}
\E \EE_0 \Ll[ \exp\Ll( -\sum_{k \in [g_n,h_n]} s_k \Rr), \sum_{k \in [g_n,h_n]} X_k \cdot \ell \ge (1-\eps)n  \Rr] \\
 \le \exp\Ll( -(1-4\eps) \Ll[\frac{dv\b}{2} + \frac{I\b}{v\b'} \Rr] n\Rr).
\end{multline}
For $\lambda\b  \ll R\b^{-1}$ to be determined, the left-hand side above is bounded by
$$
e^{-(1-\eps)\lambda\b n} \  \E\EE_0\Ll[ \exp\Ll( \sum_{k \in [g_n,h_n]}- s_k+\lambda\b \ X_k \cdot \ell \Rr)\Rr].
$$
Using Proposition~\ref{unitsX} and the Markov property, for all $\beta$ small enough, we can bound the latter by
\begin{multline*}
e^{-(1-\eps)\lambda\b n}\Ll(1 - (1-4\eps)R\b^2 \ I\b +(1 + \eps) \ \frac{\lambda\b^2 R\b^2}{2d} \Rr)^{|[g_n,h_n]|} \\
\le \exp\Ll(-(1-\eps)\lambda\b \ n +|[g_n,h_n]| \Ll(- (1-4\eps)R\b^2 \ I\b + (1 + \eps) \ \frac{\lambda\b^2 R\b^2}{2d} \Rr)  \Rr),
\end{multline*}
where we used the fact that $\log(1+x) \le x$. For $[g_n,h_n] \in \mcl{S}_n$, inequality \eqref{defEn} transfers into an inequality on the cardinality of $[g_n,h_n]$, and thus the latter is bounded by
$$
\exp\Ll(-\Ll[(1-\eps) \lambda\b  + (1-4\eps) \ \frac{I\b}{v\b'} - (1 + \eps) \ \frac{\lambda\b^2}{2d \ v\b} \Rr] n  \Rr).
$$
Choosing $\lambda\b = {dv\b(1-\eps)}/{(1+\eps)}$, we arrive at the bound
$$
\exp\Ll(-\Ll[\frac{(1-\eps)^2 d v\b}{2(1+\eps)}  + (1-4\eps) \ \frac{I\b}{v\b'} \Rr] n  \Rr),
$$
which implies the bound \eqref{goalofspeed}, and thus finishes the proof.
\end{proof}

\begin{cor}
\label{corcase1}
There exists $C > 0$ (depending on $\eps$) such that for any $\beta$ small enough, one has
$$
e_{\beta,n} \le C \exp\Ll(-(1-6\eps)\sqrt{2d\  I\b} \ n \Rr),
$$
with $I\b \ge (1-\eps)^2\  \mcl{I}\b$.
\end{cor}
\begin{proof}
The fact $I\b \ge (1-\eps)^2\  \mcl{I}\b$ was seen in \eqref{compIb}.

Let $0 = x_0 < x_1 < \ldots < x_l < x_{l+1} =+\infty$ be a subdivision of $\R_+$. The events
$$
\mcl{E}_n^c, \ \mcl{E}_n(x_0 \sqrt{I\b},x_1 \sqrt{I\b}),\  \ldots, \ \mcl{E}_n(x_l \sqrt{I\b},x_{l+1} \sqrt{I\b})
$$
form a partition of the probability space. Applying Proposition~\ref{costofspeed}, we thus get that for $\beta$ small enough, $e_{\beta,n}$ is smaller than
\begin{multline*}
2 \exp\Ll(- \eps^{-1} \sqrt{I\b}\ n\Rr)  +  2 \sum_{i = 0}^{l-1} \exp\Ll(-(1-5\eps)\Ll[\frac{dx_i}{2} + \frac{1}{x_{i+1}}\Rr] \sqrt{I\b} \ n \Rr) \\
 + 2 \exp\Ll(-(1-5\eps)\frac{dx_l}{2} \sqrt{I\b} \ n \Rr) .
\end{multline*}
When $\max_{i < l} |x_{i+1}-x_i|$ is taken small enough, and $x_l$ large enough, the dominant exponential has an exponent which we can take as close as we wish (that is, up to a multiplicative factor going to $1$) to the minimum of the function 
$$
x \mapsto (1-5\eps)\Ll[\frac{dx}{2} + \frac{1}{x}\Rr] \sqrt{I\b}.
$$ 
This minimum is $(1-5\eps)\sqrt{2d \ I\b}$, and we thus obtain the proposition.
\end{proof}

%
%
%
%
%
%%%%%%%%%%%%%%%%%%%%%%%%%%%%%%%%%%%%%%%%%%%%%%%%%%%%%%%%%%%%%%
%%%%%%%%%%%%%%%%%%%%%%%%%%%%%%%%%%%%%%%%%%%%%%%%%%%%%%%%%%%%%%
%
%
%
\section{Sites with small potential never contribute}
\label{s:intermediate}
\setcounter{equation}{0}

\begin{prop}
\label{existM}
Let $z_0 > 0$ be such that for any $z < z_0$, one has $f(z) > z/2$. Define
\begin{equation}
\label{defMb}
M\b = \eps \ \E[V \ \1_{\beta V \le z_0}].	
\end{equation}
We have $M\b \to +\infty$ as $\beta$ tends to $0$, and moreover,
$$
\int_{z \le M\b} f(\beta z) \ \d \mu(z) \le 2 \eps \ \mfk{I}\b,
$$
where we recall that $\mfk{I}\b$ was defined in \eqref{defmfkI}.
\end{prop}
\begin{proof}
The fact that $M\b \to + \infty$ is clear since we assume $\E[V] = +\infty$. For the second part, observe that, since $f(z) \le z$,
\begin{equation}
\label{exist2}
\int_{z \le M\b} f(\beta z) \ \d \mu(z) \le \beta M\b,
\end{equation}
while, by the definition of $z_0$,
\begin{equation}
\label{exist1}
2 \ \mfk{I}\b = 2 \int f(\beta z) \ \d \mu(z) \ge \beta \E[V \1_{\beta V \le z_0}].
\end{equation}
The result follows comparing \eqref{exist2} and \eqref{exist1}.
\end{proof}
From now on, we fix $a = \eps^8$. We split sites according to the value of their attached potential into three categories, according to whether the potential belongs to $[0,M\b)$, to $[M\b, \beta^{-1}a)$, or to $[\beta^{-1} a, + \infty)$. In view of Proposition~\ref{existM}, sites in the first category are always negligible (under our present assumption that $\E[V]$ is infinite). We call sites in the second category \emph{intermediate sites}. Sites in the last category are the important sites considered in the previous section. Let us write
$$
\mcl{I}\b' = \int_{M\b \le z < \beta^{-1} a} f(\beta z) \ \d \mu(z),
$$
and recall the definition of $\mcl{I}\b$ in \eqref{defmclI}. Three cases can occur.
\begin{equation}
\label{threecases}
\begin{array}{ll}
\text{Case 1:} & \mcl{I}\b' < \eps \mcl{I}\b, \\
\text{Case 2:} & \eps \mcl{I}\b \le \mcl{I}\b' < \eps^{-1} \mcl{I}\b, \\
\text{Case 3:} & \mcl{I}\b \le \eps \mcl{I}\b'.
\end{array}	
\end{equation}

Note that we may switch infinitely many times from one case to another as $\beta$ tends to $0$. If $\beta$ is sufficiently small and $\mcl{I}\b' < \eps \mcl{I}\b$, then Corollary~\ref{corcase1} gives us a sharp upper bound on $e_{\beta,n}$ (that is, an exponential bound with exponent $\sqrt{2d \ \mfk{I}\b}$, up to a multiplicative error controlled by $\eps$). In other words, Case~1 is now under control. We treat separately Cases~2 and 3 in each of the following sections.

%
%
%
%
%
%%%%%%%%%%%%%%%%%%%%%%%%%%%%%%%%%%%%%%%%%%%%%%%%%%%%%%%%%%%%%%
%%%%%%%%%%%%%%%%%%%%%%%%%%%%%%%%%%%%%%%%%%%%%%%%%%%%%%%%%%%%%%
%
%
%
\section{When intermediate and important sites both matter}
\label{s:intermediate2}
\setcounter{equation}{0}
Among the three cases displayed in \eqref{threecases}, the case when 
\begin{equation}
\label{case2}
\eps \mcl{I}\b \le \mcl{I}\b' < \eps^{-1} \mcl{I}\b,
\end{equation}
which we now investigate, is the most delicate, since both intermediate and important sites have to be taken into account.

\subsection{Splitting the (other) interval}
We want to approximate the integral $\mcl{I}\b'$ by a Riemann sum. Let us write $\rho = 1 - \eps$. We split the interval $[\beta M\b,a)$ along the subdivision given by the successive powers of $\rho$:
$$
\mcl{I}\b' \le \sum_{l = l_0}^{l\b} \int_{\beta z \in [\rho^{l},\rho^{l-1})} f(\beta z) \ \d \mu(z),
$$
where 
\begin{equation}
\label{defk0kb}
\begin{array}{l}
l_0 \text{ is the largest integer satisfying } \rho^{l_0} \ge a\text{, and} \\
l\b \text{ the smallest integer such that } \rho^{l\b - 1} < \beta M\b.	
\end{array}
\end{equation}
Since $f(z) \le z$ and $\rho = 1-\eps$, we have
\begin{equation}
\label{compmclI'}
(1-\eps) \ \mcl{I}\b' \le \sum_{l = l_0}^{l\b}  \rho^l \ \td{p}_l,
\end{equation}
where 
\begin{equation}
\label{deftdpk}
\td{p}_l = \mu\Ll( [\beta^{-1} \rho^{l}, \beta^{-1} \rho^{l-1}) \Rr).
\end{equation}
In order to lighten the notation, we keep implicit the fact that $\td{p}_l$ depends on $\beta$.

To begin with, we want to exclude the $l$'s such that $\td{p}_l$ is not much larger than $\mfk{I}\b$. Roughly speaking, we will then show that for $l$'s such that $\td{p}_l$ is indeed much larger than $\mfk{I}\b$, it is too costly to have a deviation of 
$$
\sum_{k = 0}^{T_n(\ell)-1} \beta V(S_k) \ \1_{\beta V(S_k) \in [\rho^l,\rho^{l-1})}
$$
from its typical value, so that potentials falling in $[\beta^{-1} \rho^{l}, \beta^{-1} \rho^{l-1})$ are in a ``law of large numbers'' regime.

Let 
$$
\td{\mcl{L}}\b = \Ll\{ l : l_0 \le l \le l\b \text{ and } \td{p}_l \ge \rho^{-l/2} \ \mfk{I}\b \Rr\}.
$$
%Note first that the cardinality of $\td{\mcl{L}}\b$ is bounded by $l\b$, which, in view of \eqref{defk0kb}, has a controlled growth:
%\begin{equation}
%\label{boundcardmclKb'}
%\Ll|\td{\mcl{L}}\b \Rr| \le l\b \le 2 \frac{\log(\beta M\b)}{\log(\rho)}.	
%\end{equation}
Note that
$$
\sum_{l \notin \td{\mcl{L}}\b}  \rho^l \ \td{p}_l \le \sum_{l = l_0}^{+\infty} \rho^{l/2} \ \mfk{I}\b \le  \frac{\rho^{l_0/2}}{1-\sqrt{\rho}} \ \mfk{I}\b,
$$
and since $\rho^{l_0+1} < a = \eps^8$ (see \eqref{defk0kb}),  we obtain
$$
\sum_{l \notin \td{\mcl{L}}\b}  \rho^l \ \td{p}_l \le \eps \ \mfk{I}\b,
$$
and thus
\begin{equation}
\label{compKb'}
I\b' \stackrel{\text{(def)}}{=} \sum_{l \in \td{\mcl{L}}\b}  \rho^l \ \td{p}_l \ge (1-\eps) \ \mcl{I}\b' - \eps \ \mfk{I}\b.
\end{equation}
We also define
\begin{equation}
\label{deftdI}
\td{I}\b = I\b + I\b',
\end{equation}
where $I\b$ was introduced in \eqref{defIb}. 
In view of \eqref{compIb}, \eqref{compKb'} and Proposition~\ref{existM}, we have
\begin{equation}
\label{compinterm}
\td{I}\b \ge (1-5\eps) \ \mfk{I}\b.
\end{equation}
For $l \in \td{\mcl{L}}\b$, we let
$$
\td{T}\l = \{x \in \Z^d : \beta V(x) \in [\rho^l,\rho^{l-1})\},
$$
and we call elements of $\td{T}\l$ $l$-\emph{intermediate sites}. The relevant length scale for these sites is $\td{L}\l$ defined by
$$
\td{L}\l^{-2} = \P[0 \in \td{T}\l] =  \td{p}_l \ge \rho^{-l_0/2} \ \mfk{I}\b \ge \eps^{-4} \ \mfk{I}\b.
$$
We have defined in \eqref{defLb} the length scale $L\b$ in such a way that $
L\b^{-2} \le I\b$, and recall that $I\b \le \mcl{I}\b \le \mfk{I}\b$. As a consequence, 
\begin{equation}
\label{smallerscalesaresmall}
\td{L}\l^{-2} \ge \eps^{-4} \ L\b^{-2}.	
\end{equation}
In particular, all the length scales $\td{L}\l$ associated to $l$-intermediate sites are smaller than the length scale $L\b$ we used in Section~\ref{s:lowerbd}.

\subsection{Very good boxes}
With the same $\delta$ as in Section~\ref{s:lowerbd} and for any $l \in \td{\mcl{L}}\b$, we can define the scales 
$\td{r}\l' = \td{L}\l^{2/d-\delta}$, $\td{r}\l=\td{L}\l^{2/d+\delta}$, and $\td{R}\l = \td{L}\l^{1-\delta}$. The space is partitioned into the boxes $\td{B}_{l,i} = B((2\td{R}\l+1)i,\td{R}\l)$. Each box $\td{B}_{l,i}$ is itself partitioned into boxes of size $\td{r}\l$, and we write $\td{\mcl{P}}_{l,i}$ to denote this partition. In turn, each box $b \in \td{\mcl{P}}_{l,i}$ is partitioned into boxes of size $\td{r}\l'$, and we write $\td{\mcl{P}}'_{l,b}$ to denote this partition.

Let $\td{b}' = B(x,\td{r}\l')$ be a box of size $\td{r}\l'$. Define
$$
\1_l(\td{b}') = 
\left|
\begin{array}{ll}
1 & \text{ if } \Ll| B(x,(1-\eps_1)\td{r}\l') \cap \td{T}\l \Rr | = 1, \\
0 & \text{ otherwise. }
\end{array}
\right.
$$
We say that a box $\td{B}_{l,i}$ is $l$\emph{-balanced} if for any box $\td{b} \in \td{\mcl{P}}_{l,i}$ (of size $\td{r}\l$), one has
$$
\sum_{\td{b}' \in \td{\mcl{P}}'_{l,b}} \1_l(\td{b}') \ge (1-\eps) \td{p}_l \ |\td{b}|.
$$
As the reader has noticed, this definition closely parallels the definition of a balanced box given in \eqref{defbalanced}. 

Consider a box $B_i$ of size $R\b$ (as introduced in Subsection~\ref{ss:coarse-grained}). We know from \eqref{smallerscalesaresmall} that for any $l \in \td{\mcl{L}}\b$, one has $\td{R}\l \le R\b$. As usual, we assume that we can partition the box $B_i$ into subboxes of size $\td{R}\l$, for any $l \in \td{\mcl{L}}\b$, and write this partition $\mfk{P}_{l,i}$.

We say that the box $B_i$ is \emph{very balanced} if the following two conditions hold:
\begin{enumerate}
\item the box is balanced (in the sense introduced in Subsection~\ref{ss:coarse-grained}),
\item for any $l \in \td{\mcl{L}}\b$, the proportion of $l$-balanced boxes in $\mfk{P}_{l,i}$ is at least $1-\eps^{2d}$.
\end{enumerate}
We say that the box $B_i$ is \emph{very good}  if for any $j$ such that $\|j-i\|_\infty \le 1$, the box $B_j$ is very balanced. Similarly, we say that a box of size $\td{r}\l$ is $l$\emph{-good} if the box and all its $*$-neighbours are $l$-balanced.

\subsection{Most boxes are actually very good}
We now show that, at the exponential scale, the probability that a box is very balanced is of the same order as the probability that it is simply balanced.
\begin{prop}
\label{p:verygood}
There exists $\td{c}_1 > 0$ such that for $\beta$ small enough and any $i \in \Z^d$,
$$
- \log  \P[B_i \text{ is not very balanced}] \ge  \td{c}_1 L\b^{d\delta}.
$$
\end{prop}
\begin{proof}
An estimate of the right order is provided by Proposition~\ref{p:probbadbox} for the probability that $B_i$ is not balanced. 

An examination of the proofs of Lemmas~\ref{l:subsubbox} and~\ref{p:subbox} shows that there exists $c_0 > 0$ such that, for any $l \in \td{\mcl{L}}\b$, if $\beta$ is small enough to ensure 
\begin{equation}
\label{condgood}
\td{L}_{l,\beta}^{-d\delta} \le \eps/8,
\end{equation}
then
\begin{equation}
\label{conclgood}
\P[\td{B}_{l,i} \text{ is not } l\text{-balanced}] \le \frac{|\td{B}_{l,i}|}{|\td{b}|} \ e^{-c_0 \td{L}\l^{d\delta}},
\end{equation}
where $\td{b}$ is a box of size $\td{r}\l$. Augmenting possibly the value of $c_0$, we can assume that \eqref{conclgood} holds for every $\beta$ and $l$, regardless of condition \eqref{condgood}. 

We now want to estimate 
\begin{equation}
\label{proporkbal}
- \log \P[\text{the proportion of } l\text{-balanced boxes in } \mfk{P}_{l,i} \text{ is less than }1-\eps^{2d}].
\end{equation}
The probability inside the logarithm is of the form of \eqref{bernoullige} with
$$
p = \frac{|\td{B}_{l,i}|}{|\td{b}|} \ e^{-c_0 \td{L}\l^{d\delta}},
$$
and 
$$
\eta = \eps^{2d} - p \sim \eps^{2d}.
$$
As a consequence, the quantity in \eqref{proporkbal} is larger than
\begin{eqnarray*}
|\mfk{P}_{l,i}| \psi_p(\eta) & = & \frac{|B_i|}{|\td{B}_{l,i}|}  \psi_p(\eta) \\ 
& \sim & \frac{|B_i|}{|\td{B}_{l,i}|} \eta \log(1/p) \\
& \sim & \frac{|B_i|}{|\td{B}_{l,i}|} \eps^{2d} c_0 \td{L}\l^{d\delta} \\
& \sim & 2 c \ \frac{L\b^d}{\td{L}\l^d}  \td{L}\l^{d\delta} \\
& \sim & 2 c \Ll(\frac{L\b}{\td{L}\l}\Rr)^{d-d\delta}  {L}\b^{d\delta}
\end{eqnarray*}
for some constant $c > 0$. Observe that the derivation above, which holds for $\beta$ small enough, is valid uniformly over $l$. Indeed,
\begin{equation}
\label{uniformdiv}
\td{L}_{l,\beta}^{-2} = \td{p}_l \le \mu([M\b,+\infty)) \to 0,
\end{equation}
so that all length scales go to infinity uniformly as $\beta$ tends to $0$.

Now, a union bound on
$$
\P[\exists l \in \td{\mcl{L}}\b : \text{the proportion of } l\text{-balanced boxes in } \mfk{P}_{l,i} \text{ is less than }1-\eps^{2d}],
$$
together with the fact that $\td{L}\l^{-2} = \td{p}_l \ge \rho^{l/2} \mfk{I}\b \ge \rho^{l/2} L\b^{-2}$ gives the upper bound
$$
\sum_{l = l_0}^{+\infty} \exp\Ll(-c \ \rho^{-(d-d\delta)l/4} \ L\b^{d\delta}\Rr),
$$
and thus proves the claim.
\end{proof}

Although being a very balanced box is more difficult than being simply a balanced box, Proposition~\ref{p:verygood} gives an upper bound on the probability for a box not to be very balanced which is of the same order as the upper bound we obtained in Proposition~\ref{p:probbadbox}. As a consequence, nothing changes if we replace ``good boxes'' by ``very good boxes'' throughout Section~\ref{s:lowerbd}. \textbf{From now on, we perform this replacement}: that is, whenever we refer to Section~\ref{s:lowerbd} or to objects defined therein, it is with the understanding that the ``good boxes'' appearing there are in fact ``very good boxes''.

%(* maybe useless to number this thing *)
%
%\begin{equation}
%\label{replacement}
%\begin{array}{l}
%\text{From now on, we perform this replacement: that is, when we refer to} \\
%\text{Section~\ref{s:lowerbd}, we mean that ``good boxes'' appearing there are in fact ``very} \\
%\text{good boxes''.}	
%\end{array}
%\end{equation}

\subsection{Multi-scale coarse-graining}
For any $l \in \td{\mcl{L}}\b$, we can define a coarse-grained trajectory at the scale $\td{R}\l$. For every $l \in \td{\mcl{L}}\b$, we define $j_{l,0} = 0$, and recursively,
$$
j_{l,k+1} = \inf\{ n > j_{l,k} : S_n \notin D(S_{j_{l,k}},\td{R}\l)\}.
$$
Recall the definition of $\tau$ from \eqref{deftau} as
$$
\tau = \inf\{n > 0 : S_n \notin D(S_{0},R\b)\},
$$
and let $\mcl{T}_{l}$ be the largest $k$ such that $j_{l,k} < \tau$. In words, the walk makes $\mcl{T}_{l}$ coarse-grained steps of size $\td{R}\l$ inside the ball $D(S_{0},R\b)$, and then exits this ball during the next step. We further define
$$
\mcl{T}_{l}^{(g)} = \Ll| \{ k : 0 \le k < \mcl{T}_{l} \text{ and } S_{j_{l,k}} \text{ is in a good } l\text{-box}\} \Rr|,
$$
and $\mcl{T}_{l}^{(b)} = \mcl{T}_{l} - \mcl{T}_{l}^{(g)}$. The superscripts $g$ and $b$ stand for ``good'' and ``bad'', respectively.

We start by giving a tight control of $\mcl{T}_{l}$.

\begin{prop}
\label{controlTlk} For $\beta$ small enough, the following properties hold for any $x \in \Z^d$ and $l \in \td{\mcl{L}}\b$.
\begin{enumerate}
\item 
$$
\Ll(1-\frac{\eps}{2}\Rr) \ \Ll(\frac{R\b}{\td{R}\l}\Rr)^2 \le \EE_x[\mcl{T}_{l}] \le \Ll(1+\frac{\eps}{2}\Rr) \ \Ll(\frac{R\b}{\td{R}\l}\Rr)^2.
$$		
\item For any $m \in \N$,
$$
\PP_x\Ll[ \ \mcl{T}_{l} \ge 10 m \ \EE_0[\mcl{T}_{l}] \  \Rr] \le \frac{1}{2^m}.
$$
\item There exists $\lambda_0 > 0$, $C_0>0$ independent of $\beta$ and $\eps$ such that, for any $0 \le \lambda \le \lambda_0$,
$$
\EE_x\Ll[\exp \Ll(\lambda \ \frac{\mcl{T}_{l}}{\EE_0[\mcl{T}_{l}]} \Rr)\Rr] \le C_0,
$$
$$
\EE_x\Ll[\exp \Ll(- \lambda \Ll(\frac{\mcl{T}_{l}}{\EE_0[\mcl{T}_{l}]} - 1 \Rr) \Rr)\Rr] \le \exp(C_0 \lambda^2).
$$
\end{enumerate}
\end{prop}
\begin{proof}
Note that $\EE_x[\mcl{T}_{l}]$ depends neither on $x$, nor on the environment. One can check that 
\begin{equation}
\label{e:mart}
\Ll(|S_{j_{l,k}} - S_{0}|^2 - k \td{R}_{l,\beta}^2\Rr)_{k \ge 0}
\end{equation}
is a submartingale with respect to the filtration generated by the random walk at the successive times $(j_{l,k})_{k \ge 0}$ (recall that we write $| \cdot |$ for the Euclidian norm). 
Hence, 
$$
\td{R}_{l,\beta}^2 \ \EE_x[k\wedge \mcl{T}_{l}] \le \EE_x\Ll[ |S_{j_{l,k\wedge \mcl{T}_{l}}} - S_{0}|^2 \Rr].
$$
Using the monotone convergence theorem for the left-hand side, and the dominated convergence theorem for the right-hand side, we obtain
$$
\EE_x[\mcl{T}_{l}] \le \td{R}_{l,\beta}^{-2} \ \EE_x\Ll[ |S_{j_{l,\mcl{T}_{l}}} - S_{0}|^2 \Rr].
$$
By the definition of $\mcl{T}_{l}$, we know that
$$
\Ll| |S_{j_{l,\mcl{T}_{l}}} - S_{0}| - R\b \Rr| \le \td{R}\l,
$$
so
$$
\EE_x[\mcl{T}_{l}] \le \Ll( \frac{R\b}{\td{R}\l} + 1 \Rr)^2.
$$
Now, it follows from \eqref{smallerscalesaresmall} that
$$
\td{R}\l = \td{L}\l^{1-\delta} \le \eps^{2(1-\delta)} \ L\b^{1-\delta} = \eps^{2(1-\delta)} \ R\b,
$$
or equivalently,
$$
1 \le \eps^{2(1-\delta)} \ \frac{R\b}{\td{R}\l},
$$
and this leads to the upper bound in part (1). The lower bound can be obtained in the same way. Indeed, \eqref{e:mart} fails to be a proper martingale only due to lattice effects. Taking these lattice effects into account, we observe that
$$
\Ll(|S_{j_{l,k}} - S_{0}|^2 - k (\td{R}_{l,\beta}^2 + \sqrt{d}) \Rr)_{k \ge 0}
$$
is a supermartingale. Keeping the subsequent arguments unchanged, we obtain the lower bound of part (1). 

\medskip

We have thus seen that starting from $x$, the expectation of the number of $\td{R}\l$-steps performed before exiting $D(x,R\b)$ does not exceed $(1+\eps) R\b^2/\td{R}\l^2$. Clearly, if instead we start from a point $y \in D(x,R\b)$, then the number of $\td{R}\l$-steps performed before exiting $D(x,R\b)$ is bounded by the number of steps required to exit $D(x,2R\b)$, so in particular, its expectation is bounded by 
$$
(1+\eps) 4 R\b^2/\td{R}\l^2 \le 5 \EE_0[\mcl{T}_{l}].
$$
Part (2) follows using the Markov property and Chebyshev's inequality. Part (3) is a direct consequence of part (2). (For the second part, one can use the uniform exponential integrability obtained in the first part to justify that
$$
\EE_x\Ll[\exp \Ll(- \lambda \Ll(\frac{\mcl{T}_{l}}{\EE_0[\mcl{T}_{l}]} - 1 \Rr) \Rr)\Rr] = 1+\frac{\lambda^2}{2} \EE_x\Ll[ \Ll( \frac{\mcl{T}_{l}}{\EE_0[\mcl{T}_{l}]} - 1 \Rr)^2 \Rr] + o(\lambda^2),
$$
and observe that due to part (2), the expectation on the right-hand side remains bounded uniformly over $\beta$.)
\end{proof}
With the next proposition, we have a first indication that $\mcl{T}_{l}^{(b)}$ is not very large when $S_0$ is in a very good box.
\begin{prop}
\label{p:Tb1}
For $\beta$ small enough and $l \in \td{\mcl{L}}\b$, if $x$ lies in a very good box, then
$$
\EE_x\Ll[ \mcl{T}_{l}^{(b)} \Rr] \le \eps^3 \ \EE_0[\mcl{T}_{l}].
$$	
\end{prop}
\begin{proof}
For simplicity, we assume that $x = 0$.
Under this circumstance, the box $B(0,R\b)$ is very good. In particular, the set
$$
\mcl{B} = \Ll\{ \td{B}_{l,i} \in \mfk{P}_{l,0} : \td{B}_{l,i} \text{ is not } l\text{-good} \Rr\}
$$
has cardinality at most $3^d \eps^{2d} | \mfk{P}_{l,0} |
% = \eps^{2d} R\b^d / \td{R}\l^d
$. 

We will need the observation that there exists a constant $C_1 > 0$, independent of $\beta$ and $l$, such that
\begin{equation}
\label{controlproba}
\PP_0[S \text{ visits the box } \td{B}_{l,i}] \le \frac{C_1}{ |i|^{d-2}}.
\end{equation}
This fact is 
%\cite[Proposition~6.4.2]{LL}.
\cite[Proposition 1.5.10]{law} with $m = + \infty$. For the reader's convenience, we now give a sketch of the proof. Let $G(\cdot,\cdot)$ be the Green function of $S$, and $\tau'$ be the hitting time of $\td{B}_{l,i}$. Let us also write $z_i = (2\td{R}\l+1)i$ to denote the centre of the box $\td{B}_{l,i}$. Then
$$
G(0,z_i) = \EE_0[ G(S_{\tau'},z_i) , \tau' < + \infty ],
$$
and the conclusion, that is a bound on $\PP_0[\tau' < + \infty]$, is obtained through the following estimates on the Green function:
$$
G(0,z_i) \simeq |z_i|^{2-d} \quad \text{and} \quad G(S_{\tau'},z_i) \simeq \td{R}\l^{2-d} \text{  when } \tau' < +\infty.
$$

In view of \eqref{controlproba}, it is easy to show that the expected number of visits of the $\td{R}\l$-coarse-grained random walk inside a fixed box of size $\td{R}\l$ is bounded, uniformly over $\beta$. Indeed, from this box, one has some non-degenerate probability to move at a distance a constant multiple of $\td{R}\l$ in a bounded number of $\td{R}\l$-steps, and once there, \eqref{controlproba} ensures that there is a non-degenerate probability to never go back to the box. To sum up, in order to prove the proposition, it suffices to bound the expectation of the number of $\td{R}\l$-boxes of $\mcl{B}$ that are visited by $S$.

Using \eqref{controlproba}, we get that this number is bounded by
$$
1 + \sum_{i \in \mcl{B} \setminus \{0\}} \frac{C_1}{|i|^{d-2}}.
$$
Whatever the set $\mcl{B}$ with cardinality smaller than $3^d \eps^{2d} | \mfk{P}_{l,0} |$ is, the sum above is bounded by the sum obtained when $\mcl{B}$ is a ball centred around $0$ and of radius $C \eps^{2} R\b / \td{R}\l$. Comparing this sum with an integral leads to the upper bound
$$
C \int_{1\le |z| \le C\eps^{2} R\b / \td{R}\l} |z|^{2-d} \ \d z \le C \eps^4 \Ll(\frac{R\b}{\td{R}\l}\Rr)^2.
$$
This yields the desired result, thanks to part (1) of Proposition~\ref{controlTlk}. Now, for $x \neq 0$, the same reasoning applies, the only difference being that one has to consider not only the $R\b$-box that contains $x$, but also its $*$-neighbours.
\end{proof}

\begin{prop}
\label{p:Tb2}
Let $\lambda_1 = \lambda_0/12$, where $\lambda_0$ appears in part (3) of Proposition~\ref{controlTlk}. We have
$$
\EE_0\Ll[ \exp\Ll( \lambda_1 \frac{\mcl{T}_l^{(b)}}{\EE_0[\mcl{T}_l]} \Rr) \Rr] \le \exp\Ll(\eps^{5/4}\Rr).
$$
\end{prop}
\begin{proof}
Let us write $A$ for the event
$$
\lambda_1 \frac{\mcl{T}_l^{(b)}}{\EE_0[\mcl{T}_l]} > \eps^{3/2}.
$$
Decomposing the expectation under study along the partition $\{A,A^c\}$ gives the bound
\begin{equation}
\label{e:Tb2}
\exp\Ll(\eps^{3/2}\Rr) + \EE_0\Ll[ \exp\Ll( \lambda_1 \frac{\mcl{T}_l^{(b)}}{\EE_0[\mcl{T}_l]} \Rr) , \ A \Rr].
\end{equation}
Decomposing this new expectation according to the event 
$$
A' = \Ll\{ \exp\Ll( \lambda_1 \frac{\mcl{T}_l^{(b)}}{\EE_0[\mcl{T}_l]} \Rr) > \eps^{-1/8} \Rr\},
$$
we can bound \eqref{e:Tb2} by
\begin{equation}
\label{e:Tb2-2}
\exp\Ll(\eps^{3/2}\Rr) + \eps^{-1/8} \ \PP_0[A] + \EE_0\Ll[\exp\Ll( \lambda_1 \frac{\mcl{T}_l^{(b)}}{\EE_0[\mcl{T}_l]} \Rr) , \ A' \Rr].
\end{equation}
We learn from Proposition~\ref{p:Tb1} that $\PP_0[A] \le \eps^{3/2}/\lambda_1$. Finally, the last expectation in \eqref{e:Tb2-2} is bounded by
$$
\eps^{11/8} \ \EE_0\Ll[\exp\Ll( \lambda_0 \frac{\mcl{T}_l^{(b)}}{\EE_0[\mcl{T}_l]} \Rr) \Rr],
$$
and since $\mcl{T}_l^{(b)} \le \mcl{T}_l$, the result follows from part (3) of Proposition~\ref{controlTlk}.
\end{proof}

\subsection{The cost of $l$-good steps}
Let 
$$
\td{\tau}_l = \inf\{k > 0 : S_k \notin D(S_0,\td{R}\l) \}.
$$
We would like to derive a sharp control of
\begin{equation}
\label{deftdsl}
\td{s}_l = \sum_{k = 0}^{\td{\tau}_l - 1} \beta V(S_k) \1_{\{\beta V(S_k) \in [\rho^l,\rho^{l-1})\}}.
\end{equation}
\begin{prop}
\label{laplacetds}
Assume that $\beta$ is small enough, and that $x$ lies in an $l$-good box. For every $l \in \td{\mcl{L}}\b$ and every $\gamma \ge 0$, one has
$$
\EE_x[e^{- \gamma \td{s}_l}] \le 1 -  (1-3\eps) \ \td{R}\l^2 \ \td{p}_l  \ f(\gamma \rho^l),
$$
where we recall that $f$ was defined in \eqref{deff}.
\end{prop}
\begin{proof}
The starting point is to define a quantity smaller than $\td{s}_l$ that matches the definition of $s$ given in \eqref{defs}, where the length scale $\td{L}\l$ replaces $L\b$ throughout. The analysis is then identical to the one we have performed to prove Proposition~\ref{p:laplaces}. The fact that the identity holds for $\beta$ small enough, uniformly over $l \in \td{\mcl{L}}\b$, follows again from \eqref{uniformdiv}. 
\end{proof}

\subsection{The law of large numbers regime}

We now construct a sequence of coarse-grained instants on the scale $\td{R}\l$. We do it however with a twist, since each time the walk exits a ball of radius $R\b$, the current $\td{R}\l$-step is simply discarded, and we start the $\td{R}\l$-coarse-graining afresh. More precisely, recall that $j_0',j_1',\ldots$ denote the maximal subsequence of $R\b$-coarse-graining instants such that for any $k$, $S_{j_k'}$ lies in a very good box (we may call these the instants of $R\b$-very good steps). We let
$$
\textrm{seq}_{l} = (j_{l,0}, \ldots, j_{l,\mcl{T}_l-1}),
$$
$$
\textrm{seq}_{l,k} = \textrm{seq}_l \circ \Theta_{j_k'},
$$
where $\Theta$ is the time shift.
Then $\tdj_{l,0},\tdj_{l,1},\ldots$ is obtained as the concatenation of the sequences $\textrm{seq}_{l,0},\textrm{seq}_{l,1}, \ldots$ Out of the sequence $\tdj_{l,0},\tdj_{l,1},\ldots$, we extract a maximal subsequence $\tdj_{l,0}',\tdj_{l,1}',\ldots$ such that for any $k$, $S_{\tdj_{l,k}'}$ lies in an $l$-good box (we may call this an $l$-good step). Note that all $\tdj_{l,k}'$ are stopping times (with respect to the natural filtration of $S$, for every fixed environment).

Recall that $K_n$ is such that $j_{K_n}' \le T_n(\ell)$. In words, $K_n$ is a lower bound on the number of very good $R\b$-steps performed by the walk before the time it reaches the distant hyperplane. We let
$$
\td{K}_{l,n} = \max\{k : \tdj_{l,k}' < j_{K_n}'\}.
$$
This gives us a lower bound on the number of $l$-good $\td{R}\l$-steps performed by the walk before reaching the distant hyperplane.
Let
$$
\mcl{T}_{l,k} = \mcl{T}_l \circ \Theta_{j_k'},
$$
and define similarly $\mcl{T}_{l,k}^{(g)}$, $\mcl{T}_{l,k}^{(b)}$. By definition, the number of $l$-good steps performed in the $k$-th $R\b$-very good step is $\mcl{T}_{l,k}^{(g)}$, hence
$$
\td{K}_{l,n} = \sum_{k = 0}^{K_n - 1} \mcl{T}_{l,k}^{(g)}.
$$
We also introduce
\begin{equation}
\label{deftdslk}
\td{s}_{l,k} = \td{s}_l \circ \Theta_{\tdj_{l,k}'},	
\end{equation}
where $\td{s}_{l}$ was defined in \eqref{deftdsl}.

Let $[g,h]$ be a fixed surgery. We can, out of the concatenation of $(\textrm{seq}_{l,k})_{k \in [g,h]}$, extract a maximal subsequence $(\tdj^{[g,h]}_{l,k}, 0 \le k < \td{K}^{[g,h]}_{l})$ such that for every $k$, $S_{\tdj^{[g,h]}_{l,k}}$ lies in an $l$-good box, and define
\begin{equation}
\label{relKghl}
\td{K}^{[g,h]}_{l} 
%= \max\{k : \tdj^{[g,h]}_{l,k} < j_{K_n}'\} 
= \sum_{k \in [g,h]} \mcl{T}_{l,k}^{(g)}.
\end{equation}
For $k \ge \td{K}^{[g,h]}_{l}$, we let $\tdj^{[g,h]}_{l,k} = + \infty$. The important thing to notice is that for any $k$, $\tdj^{[g,h]}_{l,k}$ is a stopping time (with respect to the natural filtration of $S$, for every fixed environment). Finally, we let 
$$
\td{s}_{l,k}^{[g,h]} = \td{s}_l \circ \Theta_{\tdj^{[g,h]}_{l,k}}.
$$

The next proposition ensures two things: first, that if $|[g,h]|$ is of order $n/(R\b^2 v\b)$, then outside of a negligible event, $\td{K}^{[g,h]}_{l}$ is at least $(1-2 \eps) n/(\td{R}\l^2 v\b)$; second, that the contribution of the $l$-intermediate sites, cut according to the surgery $[g,h]$, is in a law of large numbers regime, outside of a negligible event (recall that the average contribution of $l$-intermediate sites is of the order of $\rho^l \ \td{p}_l$).
\begin{prop}
\label{lotsoflgood}
There exists $c_6, C_2 > 0$ (depending on $\eps$) such that the following holds. Let $v\b' \ge 0$
% with $v\b' \le 2 v\b$, 
and let $[g_n,h_n]$ be a sequence of surgeries such that 
\begin{equation}
\label{hyponsurg}
\frac{n}{R\b^2 v\b'} < |[g_n,h_n]| %\le \frac{n}{R\b^2 v\b} 	
.
\end{equation}
\begin{enumerate}
\item We have
\begin{equation}
\label{e:lotsof1}
\P \PP_0\Ll[ \td{K}^{[g_n,h_n]}_{l} \le (1-2 \eps) \ \frac{n}{\td{R}\l^2 v\b'} \Rr]  \le 2 \exp\Ll( - \frac{c_6}{R\b^2 v\b'} \ n\Rr).
\end{equation}
\item
Let $E_l([g_n,h_n])$ be the event 
$$
\sum_{k = 0}^{\td{K}^{[g_n,h_n]}_{l}-1} \td{s}_{l,k}^{[g_n,h_n]} > (1-6 \eps)   \ \rho^l \ \td{p}_l\ \frac{n}{v\b'},
$$
and $E_l^c([g_n,h_n])$ be its complement. We have
\begin{equation}
\label{e:lotsof2}
\P \PP_0\Ll[ E_l^c([g_n,h_n])  \Rr] \le 2 \exp\Ll( -  \frac{\eps^2 \ \td{p}_l}{6 \ v\b'}  \ n\Rr).
\end{equation}
\item Moreover,
\begin{equation}
\label{e:lotsof3}
\P \PP_0\Ll[ \bigcup_{l \in \td{\mcl{L}}\b} E_l^c([g_n,h_n])\Rr] \le C_2  \exp\Ll( - \frac{\eps^{-2} \ \mfk{I}\b}{6 \ v\b'} \ n\Rr).
\end{equation}
\end{enumerate}
\end{prop}
\begin{proof}
For part (1), it is sufficient to show the result with a shortened surgery $[g_n',h_n']$, which coincides with $[g_n,h_n]$ for its first $n/(R\b^2 v\b')$ terms, and then stops. Considering \eqref{relKghl} and the fact that $\mcl{T}_{l,k}^{(g)} = \mcl{T}_{l,k} - \mcl{T}_{l,k}^{(b)}$, it suffices to show that the following two probabilities are sufficiently small:
\begin{equation}
\label{prob1}
\P \PP_0\Ll[ \sum_{k \in [g_n',h_n']} \mcl{T}_{l,k} \le (1- \eps) \ \frac{n}{\td{R}\l^2 v\b'} \Rr],
\end{equation}
\begin{equation}
\label{prob2}
\P \PP_0\Ll[ \sum_{k \in [g_n',h_n']} \mcl{T}_{l,k}^{(b)} \ge \eps \ \frac{n}{\td{R}\l^2 v\b'} \Rr].
\end{equation}
Let us start by examining \eqref{prob1}. With the help of part (1) of Proposition~\ref{controlTlk}, we can bound this probability by
\begin{multline*}
\P \PP_0\Ll[ \sum_{k \in [g_n',h_n']} \frac{\mcl{T}_{l,k}}{\EE_0[\mcl{T}_l]} \le (1-\frac{\eps}{3})  \ \frac{n}{R\b^2 v\b'}  \Rr] \\
\le
\P \PP_0\Ll[ \sum_{k \in [g_n',h_n']} \Ll(\frac{\mcl{T}_{l,k}}{\EE_0[\mcl{T}_l]} - 1\Rr) \le -\frac{\eps}{3}  \ \frac{n}{R\b^2 v\b'}  \Rr],
\end{multline*}
since $|[g_n',h_n']| = n/(R\b^2 v\b')$. 
%Using also the fact that $v\b'\le 2 v\b$, we can bound the latter by
%$$
%\P \PP_0\Ll[ \sum_{k \in [g_n,h_n]} \Ll(\frac{\mcl{T}_{l,k}}{\EE_0[\mcl{T}_l]} - 1\Rr) \le -\frac{\eps}{6}  \ \frac{n}{R\b^2 v\b}  \Rr].
%$$
Let $0 < \lambda \le \lambda_0$, where $\lambda_0$ is given by part (3) of Proposition~\ref{controlTlk}. The probability above is smaller than
\begin{multline*}
\exp\Ll(-\lambda\frac{\eps}{3} \frac{n}{R\b^2 v\b}\Rr) \ \E \EE_0\Ll[  \exp\Ll(-\lambda \sum_{k \in [g_n',h_n']} \Ll(\frac{\mcl{T}_{l,k}}{\EE_0[\mcl{T}_l]} - 1\Rr) \Rr)  \Rr] \\
\le \exp\Ll(-(\lambda \frac{\eps}{3} - C_0 \lambda^2) \frac{n}{R\b^2 v\b'}\Rr), 
\end{multline*}
where $C_0$ is given by part (3) of Proposition~\ref{controlTlk}, and in the last step, we used the Markov property and the fact that $|[g_n,h_n]| \le {n}/({R\b^2 v\b})$. It then suffices to take $\lambda$ small enough to get an appropriate bound on \eqref{prob1}.

We now turn to \eqref{prob2}. Using part (1) of Proposition~\ref{controlTlk}, we can bound the probability appearing there by
$$
\P \PP_0\Ll[ \sum_{k \in [g_n',h_n']} \frac{\mcl{T}_{l,k}^{(b)}}{\EE_0[\mcl{T}_l]} \ge \frac{\eps}{2} \ \frac{n}{R\b^2 v\b'} \Rr].
$$
This in turn is bounded by
$$
\exp\Ll(- \lambda_1 \frac{\eps}{2} \ \frac{n}{R\b^2 v\b'}\Rr) \ \E \EE_0\Ll[ \exp\Ll(\lambda_1 \sum_{k \in [g_n',h_n']} \mcl{T}_{l,k}^{(b)}\Rr) \Rr].
$$
As before, using the Markov property and Proposition~\ref{p:Tb2}, 
%together with the facts that $v\b' \le 2 v\b$ and $|[g_n,h_n]| \le {n}/({R\b^2 v\b})$, 
we get the bound
$$
\exp\Ll(-  (\lambda_1 \frac{\eps}{2}  - \eps^{5/4}) \ \frac{n}{R\b^2 v\b'}     \Rr),
$$
and this proves the claim (provided we fixed $\eps$ small enough).

\medskip

We now prove \eqref{e:lotsof2}. In view of part (1) and of the fact that $\td{p}_l \ll R\b^{-2}$, we can assume that 
$$
\td{K}^{[g_n,h_n]}_{l} > m_n \stackrel{\text{(def)}}{=} (1-2 \eps) \ \frac{n}{\td{R}\l^2 v\b'}.
$$
Let $\gamma \ge 0$ to be determined. On this event, the probability of $E_l^c([g_n,h_n])$ is bounded by
$$
\exp\Ll(\gamma (1-6\eps) \ \rho^l \ \td{p}_l \ \frac{n}{v\b'} \Rr) \E \EE_0\Ll[ \exp\Ll( -\gamma \sum_{k = 0}^{m_n-1} \td{s}_{l,k}^{[g_n,h_n]}  \Rr) \Rr].
$$
Using Proposition~\ref{laplacetds}, the fact that $1-x \le e^{-x}$ for $x \ge 0$, and the Markov property, we obtain the bound
$$
\exp\Ll(\gamma (1-6\eps) \ \rho^l \ \td{p}_l \ \frac{n}{v\b'}  - (1-3\eps) \ \td{R}\l^2 \ \td{p}_l  \ f(\gamma \rho^l) \ m_n \Rr).
$$
Substituting by the value of $m_n$, this transforms into
\begin{equation}
\label{eqcomput1}
\exp\Ll( -\frac{n}{v\b'} \td{p}_l \Ll[ (1-3\eps)(1-2\eps) f(\gamma \rho^l) - (1-6\eps) \gamma \rho^l \Rr]    \Rr).	
\end{equation}
A simple computation shows that 
$$
\frac{f(z)}{z} = 1-\frac{z}{q_d} + O(z^2),
$$
and since $q_d < 1$, for $\eps$ small enough, one has
$$
\frac{f(\eps/2)}{\eps/2} \ge 1-\frac{\eps}{2}.
$$
For $\gamma$ such that $\gamma \rho^l = \eps/2$, the quantity in square brackets appearing in \eqref{eqcomput1} is bounded from below by
$$
(1-3\eps)(1-2\eps)(1-\frac{\eps}{2})\frac{\eps}{2} - (1-6\eps) \frac{\eps}{2} \ge \frac{\eps^2}{6},
$$
and this proves \eqref{e:lotsof2}.

\medskip

We now examine how to go from \eqref{e:lotsof2} to \eqref{e:lotsof3}. We will actually show that
\begin{equation}
\label{e:lotsof4}
\P \PP_0\Ll[ \bigcup_{l \in \td{\mcl{L}}\b} E_l^c([g_n.h_n])\Rr] \le 2 \Ll( 1 + \frac{6 \ v\b'}{n \ \eps^2 \ \mfk{I}\b} \Rr) \ \exp\Ll( -n \ \frac{\eps^{-2} \ \mfk{I}\b}{6 \ v\b'} \Rr).
\end{equation}
Since the probability is always bounded by $1$, proving \eqref{e:lotsof4} is sufficient. For greater clarity, let us write 
$$
\mfk{c} = \frac{n \ \eps^2 \ \mfk{I}\b}{6 \ v\b'}.
$$
Since for $l \in \td{\mcl{L}}\b$, one has $\td{p}_l \ge \rho^{-l/2} \mfk{I}\b$, it suffices to bound
$$
\exp\Ll(-\mfk{c} \  \rho^{-l_0/2}\Rr) +  \sum_{l = l_0+1}^{+\infty} \exp\Ll(-\mfk{c} \  \rho^{-l/2}\Rr).
$$
The sum above is smaller than
$$
\int_{l_0}^{+\infty} \exp\Ll(-\mfk{c} \  \rho^{-l/2}\Rr) \ \d l = \int_{\rho^{-l_0/2}}^{+\infty} \exp\Ll(-\mfk{c} \  u\Rr) \ \frac{2 \ \d u}{u \ \log(1/\rho)}.
$$
For $u \ge \rho^{-l_0/2} \ge \eps^{-4}$, one has $u \log(1/\rho) \ge 2$, so the latter integral is bounded by
$$
\frac{1}{\mfk{c}} \exp\Ll( -\mfk{c} \  \rho^{-l_0/2} \Rr),
$$
and since $\rho^{-l_0/2} \ge \eps^{-4}$, this proves \eqref{e:lotsof4}, and thus also \eqref{e:lotsof3}.
\end{proof}

\subsection{Speeds and their costs}
\begin{prop}
\label{closetothere}
There exists $C_3 > 0$ such that the following holds. Let $v\b < v\b'$ satisfy 
\begin{equation}
\label{condonspeeds}
v\b' \le \eps^{-1} \sqrt{\mfk{I}\b}.
\end{equation}
For any $\beta$ small enough, if \eqref{case2} holds, then
\begin{equation}
\label{eq:costofspeed2}
e_{\beta,n}(v\b,v\b') \le C_3 \exp\Ll(-(1-7\eps)\Ll[\frac{dv\b}{2} + \frac{\td{I}\b}{v\b'}\Rr] \ n \Rr) ,
\end{equation}
where $e_{\beta,n}(v\b,v\b')$ was defined in \eqref{defenabv} and $\td{I}\b$ in \eqref{deftdI}. 
\end{prop}
\begin{proof}
An upper bound on $e_{\beta,n}(v\b,v\b')$ is given by
$$
\E \EE_0 \Ll[ \exp\Ll( -\sum_{l\in \td{\mcl{L}}\b} \sum_{k=0}^{\td{K}^{[G',H']}_{l}-1} \td{s}_{l,k}^{[G',H']} -\sum_{k \in [G',H']} s_k \Rr), \mcl{E}_n(v\b,v\b') \Rr].
$$
Indeed, this corresponds to our partition of sites into $l$-intermediate sites (with $l \in \td{\mcl{L}}\b$) and important sites, where we drop certain contributions according to the surgery and the coarse-graining.

Let us first see that one can find $c'$ such that when both \eqref{condonspeeds} and \eqref{case2} hold, one has $v\b' \le c' \sqrt{I\b}$. This is true since $\mfk{I}\b \simeq \td{I}\b = I\b + I\b'$, and $I\b' \simeq \mcl{I}\b' \le \eps^{-1} \mcl{I}\b$ under the assumption \eqref{case2}, and finally $\mcl{I}\b \simeq I\b$.

Hence, as in the proof of Proposition~\ref{costofspeed}, Proposition~\ref{controlcuts} ensures that it suffices to show that, for any sequence of cuts $[g_n,h_n]$ satisfying 
$$
\frac{n}{R\b^2 v\b'} < |[g_n,h_n]| \le \frac{n}{R\b^2 v\b},
$$
one has
\begin{multline}
\label{goalofspeed2}
\E \EE_0 \Ll[ \exp\Ll( -\sum_{l\in \td{\mcl{L}}\b} \sum_{k=0}^{\td{K}^{[g_n,h_n]}_{l}-1} \td{s}_{l,k}^{[g_n,h_n]}-\sum_{k \in [g_n,h_n]} s_k \Rr), \sum_{k \in [g_n,h_n]} X_k \cdot \ell \ge (1-\eps)n  \Rr] \\
 \le C_3 \exp\Ll(-(1-6\eps)\Ll[\frac{dv\b}{2} + \frac{\td{I}\b}{v\b'}\Rr] \ n \Rr).
\end{multline}
Consider the bound given by part (2) of Proposition~\ref{lotsoflgood} on the probability of the event
$$
\bigcup_{l \in \td{\mcl{L}}\b} E_l^c([g_n.h_n]).
$$
In order to see that this bound is smaller than the right-hand side of \eqref{goalofspeed2}, it is enough to check that
$$
\frac{\eps^{-2} \ \mfk{I}\b}{6 v\b'} \ge \frac{d v\b'}{2} + \frac{\td{I}\b}{v\b'},
$$
since $v\b \le v\b'$. The infimum over all possible values of the right-hand side above is 
$\sqrt{2 d \ \td{I}\b}$. It thus suffices to observe that
$$
v\b' \le \frac{\eps^{-2} \ \mfk{I}\b}{6 \sqrt{2 d \ \td{I}\b}}.
$$
This is true under condition \eqref{condonspeeds} since $\td{I}\b \le \mfk{I}\b$.

We can thus evaluate the expectation in the left-hand side of \eqref{goalofspeed2}, restricted on the event
\begin{equation}
\label{eventlgood}
\bigcap_{l \in \td{\mcl{L}}\b} E_l([g_n.h_n]).
\end{equation}
On this event, by definition, the contributions of $l$-intermediate sites is bounded from below by a deterministic quantity. More precisely, part (2) of Proposition~\ref{lotsoflgood} ensures that the expectation in the left-hand side of \eqref{goalofspeed2}, once restricted on the event \eqref{eventlgood}, is smaller than
$$
\exp\Ll(-(1-6\eps) \frac{I\b'}{v\b'}\ n\Rr) \E \EE_0 \Ll[ \exp\Ll( -\sum_{k \in [g_n,h_n]} s_k \Rr), \sum_{k \in [g_n,h_n]} X_k \cdot \ell \ge (1-\eps)n  \Rr].
$$
where we recall that $I\b'$ was defined in \eqref{compKb'}. The remaining expectation is the same as the one met during the proof of Proposition~\ref{costofspeed}. One can thus follow the same reasoning  (since we have checked that $v\b' \le c' \sqrt{I\b}$ at the beginning of this proof, it is also true that $v\b \ll R\b^{-1}$), and arrive at the conclusion, since $I\b' + I\b = \td{I}\b$.
\end{proof}

\begin{cor}
\label{corcase2}
There exists $C > 0$ (depending on $\eps$) such that for any $\beta$ small enough, if \eqref{case2} holds, then
$$
e_{\beta,n} \le C \exp\Ll(-(1-8\eps)\sqrt{2d\  \td{I}\b} \ n \Rr),
$$
with $\td{I}\b \ge (1-5\eps) \ \mfk{I}\b$.
\end{cor}
\begin{proof}
The fact that $\td{I}\b \ge (1-5\eps) \ \mfk{I}\b$ was seen in \eqref{compinterm}. The proof of the estimate is similar to that of Corollary~\ref{corcase1}. Let $0 = x_0 < x_1 < \ldots < x_l < x_{l+1} =+\infty$ be a subdivision of $\R_+$, with $x_l = \eps^{-1}$.  Note that since $\td{I}\b \le \mfk{I}\b$, we have indeed $x_l \sqrt{\td{I}\b} \le \eps^{-1} \sqrt{\mfk{I}\b}$. We can thus apply Proposition~\ref{closetothere} with $v\b' = x_i\sqrt{\td{I}\b}$ for any $i \le l$.

The events
$$
\mcl{E}_n\Ll(x_0 \sqrt{\td{I}\b},x_1 \sqrt{\td{I}\b}\Rr),\  \ldots, \ \mcl{E}_n\Ll(x_l \sqrt{\td{I}\b},x_{l+1} \sqrt{\td{I}\b}\Rr), \ \mcl{E}_n^c,
$$
form a partition of the probability space. We decompose the expectation defining $e_{\beta,n}$ according to this partition. We use Proposition~\ref{costofspeed} to evaluate the two last terms thus obtained, and Proposition~\ref{closetothere} to evaluate the other terms, thus obtaining the bound 
\begin{multline}
\label{sumofexps}
2 \exp\Ll(- \eps^{-1} \sqrt{I\b}\ n\Rr)  
  + \exp\Ll(-(1-5\eps)\frac{dx_l}{2} \sqrt{I\b} \ n \Rr) \\
+ C_3 \sum_{i = 0}^{l-1} \exp\Ll(-(1-7\eps)\Ll[\frac{dx_i}{2} + \frac{1}{x_{i+1}}\Rr] \sqrt{\td{I}\b} \ n \Rr).
\end{multline}
The use of Proposition~\ref{costofspeed} to evaluate the expectation restricted on the event $\mcl{E}_n\Ll(x_l \sqrt{\td{I}\b},+\infty\Rr)$ is legitimate since when \eqref{case2} holds, one has $\sqrt{\td{I}\b} \simeq \sqrt{\td{I}\b} \simeq L\b^{-1} \ll R\b^{-1}$.

When $\max_{i < l} |x_{i+1}-x_i|$ is taken small enough, the dominant exponential in the sum over $i$ that appears in \eqref{sumofexps} has an exponent which we can take as close as we wish to $(1-7\eps)\sqrt{2d \ \td{I}\b}$. The first two exponentials appearing in \eqref{sumofexps} are negligible, since as explained above, we have $I\b \ge \eps \td{I}\b/4$, and remember that we chose $x_l = \eps^{-1}$.
\end{proof}

%
%
%
%
%
%%%%%%%%%%%%%%%%%%%%%%%%%%%%%%%%%%%%%%%%%%%%%%%%%%%%%%%%%%%%%%
%%%%%%%%%%%%%%%%%%%%%%%%%%%%%%%%%%%%%%%%%%%%%%%%%%%%%%%%%%%%%%
%
%
%
\section{When only intermediate sites matter}
\label{s:intermediate_only}
\setcounter{equation}{0}
We now proceed to examine the case when only intermediate sites contribute to the integral, that is, when
\begin{equation}
\label{case3}
\mcl{I}\b \le \eps \ \mcl{I}\b'.
\end{equation}
This case is a minor adaptation of the arguments developed in Section~\ref{s:intermediate2}. Since in this regime, important sites are negligible, it will be harmless to redefine their associated length scale to be
\begin{equation}
\label{newdefLb}
L\b^{-2} = \mfk{I}\b.
\end{equation}
By doing so, the most important change is that the length scale $L\b$ is no longer tied with $I\b$ (that is, we do not have $L\b^{-2} \simeq I\b$). Relation~\eqref{smallerscalesaresmall} is preserved. We say that a box $B_i$ (of size $R\b = L\b^{1-\delta}$) is \emph{locally balanced} if for any $l \in \td{\mcl{L}}\b$, the proportion of $l$-balanced boxes in $\mfk{P}_{l,i}$ is at least $1-\eps^{2d}$. Clearly, the notion of a locally balanced box is weaker than the one of a very balanced box. As before, we say that a box $B_i$ is \emph{locally good} if for any $j$ such that $\|i-j\|_\infty \le 1$, the box $B_j$ is locally balanced. 

When we refer to previous sections, it is now with the understanding that ``good'' or ``very good'' boxes are replaced by ``locally good'' boxes. 

We will now see that techniques developed previously can handle the situation under consideration without additional complication. We start by discarding the possibility of slow motions. Recall that in our new setting, $K_n$ is a lower bound on the number of $R\b$-locally good steps performed by the random walk. The property we want to prove is analogous to the one obtained in Proposition~\ref{noslowmotion}.

\begin{prop}
\label{p:inter:noslow}
For any $c_7 > 0$ and any $\beta$ small enough,
\begin{equation}
\label{e:inter:noslow}
\E \EE_0 \Ll[ \exp\Ll(- \sum_{k = 0}^{T_n(\ell)-1} \beta V(S_k)\Rr), \ K_n \ge \frac{n}{c_7 \ R\b^2 \sqrt{I\b'}} \Rr] \le 2 \exp\Ll( - \frac{\sqrt{I\b'}}{2c_7} \ n \Rr).
\end{equation}
\end{prop}
\begin{proof}
Let $v\b' = c_7 \sqrt{I\b'}$, and $E_{l,n}$ be the event
$$
\sum_{k = 0}^{\td{K}_{l,n}-1} \td{s}_{l,k} > (1-6 \eps)   \ \rho^l \ \td{p}_l\ \frac{n}{v\b'}.
$$
It is a consequence of Proposition~\ref{lotsoflgood} that
\begin{eqnarray*}
%\label{e:lotsof3bis}
\P \PP_0\Ll[ \bigcup_{l \in \td{\mcl{L}}\b} E_{l,n}^c, \ K_n \ge \frac{n}{R\b^2 v\b'}\Rr] & \le & C_2  \exp\Ll( - \frac{\eps^{-2} \ \mfk{I}\b}{6 \ v\b'} \ n\Rr) \\
& \le & C_2 \exp\Ll( - \frac{\eps^{-2} \ \sqrt{I\b'}}{6 \ c_7 } \ n\Rr).
\end{eqnarray*}
This is negligible compared to the right-hand side of \eqref{e:inter:noslow}. When one computes the expectation on the left-hand side of \eqref{e:inter:noslow} restricted to the conjunction of events
$$
\bigcap_{l \in \td{\mcl{L}}\b} E_{l,n}, \ K_n \ge \frac{n}{R\b^2 v\b'},
$$
one finds the upper bound
$$
\exp\Ll(- (1-6 \eps) \frac{I\b'}{c_7 \sqrt{I\b'}}\ n\Rr) \le \exp\Ll(-  \frac{\sqrt{I\b'}}{2 \ c_7 }\ n\Rr),
$$
and we thus obtain the result.
\end{proof}
From now on, we fix
\begin{equation}
\label{defc7}
c_7 = \eps/2,	
\end{equation}
so that the cost associated to the event
\begin{equation}
\label{inter:noslow}
K_n < \frac{n}{c_7 \ R\b^2 \sqrt{I\b'}}
\end{equation}
is much larger than $\sqrt{I\b'}$. 
\begin{prop}
\label{controlcuts2}
Recall that $\mcl{A}_n$ is the event defined in \eqref{defAn}. There exists $c_5 > 0$ such that the following holds. For any $\beta$ small enough and satisfying \eqref{case3}, the cardinality of the set of surgeries $[g,h]$ such that
$$
\P \PP_0\Ll[ [G',H'] = [g,h],  \ \mcl{A}_n \text{ and  \eqref{inter:noslow} are both satisfied}\Rr] > 0
$$
is bounded by
$$
\exp\Ll(c_5 \ L\b^{-\delta/4} \sqrt{I\b'} \ n\Rr).
$$
\end{prop}
\begin{proof}
The proof is similar to that of Proposition~\ref{controlcuts}. The important point is to notice that when \eqref{case3} holds, one has $L\b^{-2} \simeq I\b'$ (see \eqref{newdefLb}).
\end{proof}
We change the definition of the event $\mcl{E}_n(v\b,v\b')$ to the following:
\begin{equation}
\label{defEn2}
\begin{array}{c}
\displaystyle{\frac{n}{R\b^2 v\b'} \le |[G',H']| < \frac{n}{R\b^2 v\b},} \\
\displaystyle{\mcl{A}_n \text{ and  \eqref{inter:noslow} hold}}
\end{array}
\end{equation}
(compare this definition with \eqref{defEn}). We write $\mcl{E}_n^c$ for the complement of the event $\mcl{E}_n(0,+\infty)$, that is, for the event when either $\mcl{A}_n$ or \eqref{inter:noslow} fails to hold. We define $e_{\beta,n}(v\b,v\b')$ and $e_{\beta,n}^c$ as in \eqref{defenabv} and \eqref{defenc}, respectively.

\begin{prop}
\label{costofspeed2}
\begin{enumerate}
\item For $\beta$ small enough, one has
$$
e_{\beta,n}^c \le 2 \exp\Ll(-\eps^{-1} \sqrt{I\b'} \ n \Rr).
$$
\item Let $v\b$ satisfy $v\b \ll R\b^{-1}$. For any small enough $\beta$, if \eqref{case3} holds, then
$$
\P\PP_0\Ll[\mcl{E}_n(v\b,+\infty)\Rr] \le 2 \exp\Ll( -(1-5\eps) \frac{d v\b}{2} \ n \Rr).
$$
\item Let $v\b < v\b'$ satisfy $v\b' \le \eps^{-1} \sqrt{\mfk{I}\b}$. For any small enough $\beta$, if \eqref{case3} holds, then
$$
e_{\beta,n}(v\b^1,v\b^2) \le C_3 \exp \Ll( -(1-7\eps) \Ll[ \frac{d v\b}{2} + \frac{I\b'}{v\b'} \Rr] \ n \Rr).
$$
\end{enumerate}
\end{prop}
\begin{proof}
The proof of part (1) is identical to the proof of part (1) of Proposition~\ref{costofspeed}, except that we use Proposition~\ref{p:inter:noslow} instead of Proposition~\ref{noslowmotion}. Part (2) is obtained in the same way as part (2) of Proposition~\ref{noslowmotion}, while part (3) is proved as Proposition~\ref{closetothere}.
\end{proof}

\begin{cor}
\label{corcase3}
There exists $C > 0$ such that for any $\beta$ small enough, if \eqref{case3} holds, then
\begin{equation}
\label{e:corcase3}
e_{\beta,n} \le C \exp\Ll(-(1-8\eps) \sqrt{2d\ I\b'} \ n\Rr),
\end{equation}
with $I\b' \ge (1-6\eps)\ \mfk{I}\b$.
\end{cor}
\begin{proof}
In order to see that $I\b' \ge (1-6\eps)\ \mfk{I}\b$, we recall first from \eqref{compKb'} that
\begin{equation}
\label{randomtruc}
I\b' \ge (1-\eps) \ \mcl{I}\b' - \eps\  \mfk{I}\b.
\end{equation}
From Proposition~\ref{existM}, we know that
$$
(1-2\eps) \ \mfk{I}\b \le \mcl{I}\b + \mcl{I}\b',
$$
Using also \eqref{case3}, we thus obtain
$$
(1-2\eps) \ \mfk{I}\b \le (1+\eps) \ \mcl{I}\b' \stackrel{\eqref{randomtruc}}{\le} \frac{1+\eps}{1-\eps} \ \Ll( I\b'+\eps \ \mfk{I}\b \Rr),
$$
and the announced result follows, since $\eps$ can be chosen arbitrarily small.

The proof of the upper bound \eqref{e:corcase3} is identical to that of Corollary~\ref{corcase2}, except that it is now Proposition~\ref{costofspeed2} that provides us with the necessary estimates.
\end{proof}

%
%
%
%
%
%%%%%%%%%%%%%%%%%%%%%%%%%%%%%%%%%%%%%%%%%%%%%%%%%%%%%%%%%%%%%%
%%%%%%%%%%%%%%%%%%%%%%%%%%%%%%%%%%%%%%%%%%%%%%%%%%%%%%%%%%%%%%
%
%
%
\section{Extensions and link with Green functions}
\label{s:extension}
\setcounter{equation}{0}

We begin this last section by extending Theorem~\ref{t:estimen} to cases when the potential itself may depend on $\beta$. To this end, we consider for every $\beta$ small enough, a family of i.i.d.\ random variables $(V\b(x))_{x \in \Z^d}$ under the measure $\P$, whose common distribution will be written $\mu\b$. We are now interested in
\begin{equation*}
\ov{e}_{\beta,n} = \E\EE_0\Ll[ \exp\Ll( -\sum_{k = 0}^{T_n(\ell)-1} \beta V\b(S_k) \Rr)  \Rr],
\end{equation*}
and the integral
$$
\ov{\mfk{I}}\b := \int{f}(\beta z) \ \d \mu\b(z)
$$
will play the role the integral $\mfk{I}\b$ had previously.
\begin{thm}
\label{t:ext}
Let $\eps > 0$. If 
\begin{equation}
\label{hyp1}
\beta V\b(0) \xrightarrow[\beta \to 0]{\text{prob.}} 0,
\end{equation}
then for any $a > 0$, there exists $C > 0$ such that for any $\beta$ small enough and any~$n$,
$$
\ov{e}_{\beta,n} \le C \exp\Ll( -(1-\eps) \sqrt{2d \ \int_{\beta z > a} {f}(\beta z) \ \d \mu\b(z)} \ n\Rr).
$$	
If moreover,
\begin{equation}
\label{hyp2}
\ov{M}\b :=  \E\Ll[V\b \ \1_{\beta V\b \le z_0}\Rr] \xrightarrow[\beta \to 0]{} + \infty,
\end{equation}
where $z_0$ is as in Proposition~\ref{existM}, and for every $\eta > 0$,
\begin{equation}
\label{hyp3}
\mu\b \Ll( [\eta \ov{M}\b, + \infty) \Rr) \xrightarrow[\beta \to 0]{} 0,
\end{equation}
then there exists $C > 0$ such that for any $\beta$ small enough and any $n$,
$$
\ov{e}_{\beta,n} \le C \exp\Ll( -(1-\eps) \sqrt{2d \ \ov{\mfk{I}}\b} \ n\Rr).
$$
\end{thm}
\begin{proof}
The first part is an adaptation of the results of Section~\ref{s:lowerbd}. The only difference is that we replace the measure $\mu$ by $\mu\b$ everywhere. We need to check that the reference length-scale $L\b$ goes to infinity as $\beta$ tends to $0$. For this, it suffices to verify that for every $a > 0$, 
$$
\mu\b\Ll( [\beta^{-1} a, + \infty) \Rr) \xrightarrow[\beta \to 0]{} 0,
$$
and this is a consequence of the assumption in \eqref{hyp1}. 

The second part is obtained in the same way, but following Sections~\ref{s:intermediate} to \ref{s:intermediate_only}. Assuming \eqref{hyp2} enables to proceed through Section~\ref{s:intermediate}, and the condition in \eqref{hyp3} ensures that the scales of reference $\td{L}_{l,\beta}$ go to infinity as $\beta$ tends to $0$, uniformly over $l$ (see for instance \eqref{uniformdiv}).
\end{proof}
\begin{cor}
\label{cor:ext}
If
\begin{equation}
	V\b \xrightarrow[\beta \to 0]{\text{law}} V
\end{equation}
and $\E[V] = +\infty$, then for every $\eps > 0$, there exists $C > 0$ such that 
$$
\ov{e}_{\beta,n} \le C \exp\Ll( -(1-\eps) \sqrt{2d \ \ov{\mfk{I}}\b} \ n\Rr).
$$
\end{cor}
\begin{proof}
	It suffices to apply Theorem~\ref{t:ext}, checking that the conditions displayed in \eqref{hyp1}, \eqref{hyp2} and \eqref{hyp3} are satisfied. Since we assume that $V\b$ converges in law, the condition in \eqref{hyp1} is clear. For the same reason, \eqref{hyp3} is clear if we can prove that \eqref{hyp2} holds. In order to check \eqref{hyp2}, we may appeal to Skorokhod's representation theorem, which provides us with random variables $\ov{V}\b$ which are distributed as $V\b$ for each fixed $\beta$, and which converge almost surely to $\ov{V}$ distributed as $V$. For convenience, we may assume that these random variables are defined with respect to the measure $\P$. Now, by Fatou's lemma,
\begin{eqnarray*}
\liminf_{\beta \to 0} \ov{M}\b & = & \liminf_{\beta \to 0} \E\Ll[\ov{V}\b \ \1_{\beta \ov{V}\b \le z_0}\Rr] \\
& \ge & \E\Ll[ \liminf_{\beta \to 0} \ov{V}\b \ \1_{\beta \ov{V}\b \le z_0}\Rr] \\
& = & \E\Ll[ \ov{V} \Rr] = \E\Ll[V\Rr] = + \infty,
\end{eqnarray*}
and this finishes the proof.
\end{proof}

There are simple relations linking $\ov{e}_{\beta,n}$ to the average of the Green function (in the probabilistic sense) $g\b(x,y)$ defined by
$$
g\b(x,y) = \sum_{n = 0}^{+\infty} \EE_x\Ll[ \exp\Ll( -\sum_{k = 0}^{n} \beta V\b(S_k) \Rr) \1_{S_n = y} \Rr].
$$
We refer to \cite[(10)]{zer} for precise statements (note that since we consider only $d \ge 3$, the function $g\b$ is uniformly bounded). As was noted in \cite[Proposition~2]{zer}, if we define
\begin{equation}
\label{defVb}
V\b = \beta^{-1} \log(\beta V + 1),	
\end{equation}
then the function $g\b$ is also the Green function of the operator $H\b = -\triangle + \beta V$ (in the usual sense of $H\b^{-1} \delta_y$), where $\triangle$ was defined in \eqref{deftriangle}. Assuming that $\E[V] = + \infty$, the conditions of Corollary~\ref{cor:ext} are satisfied. Hence, for any $\eps > 0$, there exists $C > 0$ such that
\begin{equation}
\label{decayg}
\E[g\b(0,n\ell)] \le C \exp\Ll( -(1-\eps) \sqrt{2d \ \ov{\mfk{I}}\b} \ n\Rr).
\end{equation}
(In fact, a stronger statement can be derived from Corollary~\ref{cor:ext} by replacing the Green function by a ``point-to-hyperplane'' version of it.) Moreover, recall that
$$
\ov{\mfk{I}}\b = \int f(\beta z) \ \d \mu\b(z),
$$
where $\mu\b$ is the distribution of $V\b$ defined in \eqref{defVb}. By a change of variables, we can rewrite the integral $\ov{\mfk{I}}\b$ as
$$
\int f(\log(\beta z + 1) ) \ \d \mu(z),
$$
where $\mu$ is the distribution of $V$. Using the definition of $f$ in \eqref{deff}, we obtain that this integral is equal to that displayed in \eqref{conj:int}. The decay of the Green function as given in \eqref{decayg} is the signature of Lifshitz tails and of localized eigenfunctions for energies smaller than $\ov{\mfk{I}}\b$ (see \cite{klopp}), and hence our conjecture.


\begin{thebibliography}{99}

\bibitem[AM93]{am} M.\ Aizenman, S.\ Molchanov. Localization at large disorder and at extreme energies: an elementary derivation. \emph{Comm.\ Math.\ Phys.}\ \textbf{157} (2), 245–278 (1993).

\bibitem[ASFH01]{afhs} M.\ Aizenman, J.H.\ Schenker, R.M.\ Friedrich, D.\ Hundertmark. Finite-volume frac\-tional-moment criteria for Anderson localization. \emph{Comm.\ Math.\ Phys.}\ \textbf{224} (1), 219–253 (2001).

\bibitem[B\v{C}07]{bc} G.\ Ben Arous, J.\ \v{C}ern\'y. Scaling limit for trap models on $\Z^d$. \emph{Ann.\ Probab.}\ \textbf{35} (6), 2356-2384 (2007).

%\bibitem[B\v{C}M06]{bcm} G.\ Ben Arous, J.\ \v{C}ern\'y, T.\ Mountford. Aging in two-dimensional Bouchaud's model. \emph{Probab.\ Theory Related Fields} \textbf{134} (1), 1-43 (2006). 

%\bibitem[BG90]{bougeo} J.-P.~Bouchaud, A.~Georges. Anomalous diffusion in disordered media~: statistical mechanisms, models and physical applications. \emph{Phys. Rep.} \textbf{195} (4-5), 127-293 (1990).

\bibitem[DK89]{dk} H.\ von Dreifus, A.\ Klein. A new proof of localization in the Anderson tight binding model. \emph{Comm.\ Math.\ Phys.}\ \textbf{124} (2), 285–299 (1989).

\bibitem[Fl07]{flu1} M.~Flury. Large deviations and phase transition for random walks in random nonnegative potentials. \emph{Stochastic Process. Appl.} \textbf{117} (5), 596–612 (2007).

\bibitem[Fl08]{flu2} \bysame. Coincidence of Lyapunov exponents for random walks in weak random potentials. \emph{Ann. Probab.} \textbf{36} (4), 1528–1583 (2008). 

\bibitem[FMSS85]{fmss} J.\ Fröhlich, F.\ Martinelli, E.\ Scoppola, T.\ Spencer. Constructive proof of localization in the Anderson tight binding model. \emph{Comm.\ Math.\ Phys.}\ \textbf{101} (1), 21–46 (1985).

\bibitem[FS83]{fs} J.\ Fröhlich, T.\ Spencer. Absence of diffusion in the Anderson tight binding model for large disorder or low energy. \emph{Comm.\ Math.\ Phys.}\ \textbf{88} (2), 151–184 (1983).

\bibitem[IV12a]{iv} D.~Ioffe, Y.~Velenik. Crossing random walks and stretched polymers at weak disorder. \emph{Ann.\ Probab.}\ \textbf{40}, 714-742 (2012).
%
%\bibitem[IV12b]{ivreview} \bysame. Stretched polymers in random environment. \emph{Probability in complex physical systems}, J.-D.\ Deuschel et al.\ (Eds), Springer proceedings in mathematics \textbf{11}, 339-369 (2012).

\bibitem[IV12b]{iv12} \bysame. Self-attractive random walks: the case of critical drifts. \emph{Comm.\ Math.\ Phys.},\ to appear (2012).

\bibitem[Kl02]{klopp} F.\ Klopp. Weak disorder localization and Lifshitz tails. \emph{Comm.\ Math.\ Phys.}\ \textbf{232} (1), 125–155 (2002). 

\bibitem[KM12]{km} E.\ Kosygina, T.\ Mountford. Crossing velocities for an annealed random walk in a random potential. \emph{Stochastic Process.\ Appl.}\ \textbf{122} (1), 277–304 (2012).

\bibitem[KMZ12]{kmz} E.\ Kosygina, T.\ Mountford, M.P.W.\ Zerner. Lyapunov exponents of Green’s functions for random potentials tending to zero. \emph{Probab. Theory Related Fields}, to appear (2012).

\bibitem[La]{law} G.\ Lawler. \emph{Intersections of random walks}.
Probability and its applications, Birkh\"auser (1991).

%\bibitem[LL]{LL} G.\ Lawler, V.\ Limic

\bibitem[Mo12]{shape} J.-C.\ Mourrat. Lyapunov exponents, shape theorems and large deviations for the random walk in random potential. \emph{ALEA Lat.\ Am.\ J.\ Probab.\ Math.\ Stat.}\ \textbf{9}, 165-211 (2012).

\bibitem[Ru11]{ru} J.\ Rue\ss. Lyapunov exponents of Brownian motion: decay rates for scaled Poissonian potentials and bounds. Preprint, arXiv:1101.3404v1 (2011). 

\bibitem[Sz95]{sz} A.-S.~Sznitman. Crossing velocities and random lattice animals. \emph{Ann. Probab.} \textbf{23} (3), 1006–1023 (1995). 

%\bibitem{wu} Wüthrich - Geodesics and crossing brownian motion... regarde la limite $\beta \to +\infty$.

\bibitem[Wa01a]{wang1} W.-M.\ Wang. Mean field bounds on Lyapunov exponents in $\Z^d$ at the critical energy. \emph{Probab. Theory Related Fields} \textbf{119} (4), 453–474 (2001). 

\bibitem[Wa01b]{wangloc} \bysame. Localization and universality of Poisson statistics for the multidimensional Anderson model at weak disorder. \emph{Invent. Math.} \textbf{146} (2), 365–398 (2001). 

\bibitem[Wa02]{wang2} \bysame. Mean field upper and lower bounds on Lyapunov exponents. \emph{Amer. J. Math.} \textbf{124} (5), 851–878 (2002). 

\bibitem[Ze98]{zer} M.P.W.\ Zerner. Directional decay of the Green's function for a random nonnegative potential on ${\bf Z}^d$. \emph{Ann. Appl. Probab.} \textbf{8} (1), 246–280 (1998).

\bibitem[Zy09]{zyg1} N.~Zygouras. Lyapounov norms for random walks in low disorder and dimension greater than three. \emph{Probab. Theory Related Fields} \textbf{143} (3-4), 615–642 (2009). 

\bibitem[Zy12]{zyg2} \bysame. Strong disorder in semidirected random polymers. \emph{Ann.\ Inst.\ Henri Poincaré Probab.\ Stat.},\ to appear (2012).

\end{thebibliography}
\end{document}